\documentclass[12pt, reqno]{amsart}
\usepackage{amssymb,amsmath,amsthm,newlfont,color}
  \usepackage{psfrag}
\usepackage[T1]{fontenc}
\usepackage{a4wide}

\usepackage[utf8]{inputenc}

\newcommand{\Ll}{\mathop{\textmd{L}}}

\numberwithin{equation}{section}
\newcommand{\N}{\mathbb{N}}
\newcommand{\R}{\mathbb{R}}
\newcommand{\C}{\mathbb{C}}
\newcommand{\D}{\mathbb{D}}

\newcommand{\cF}{{\mathcal{F}}}

\newcommand{\cC}{{\mathcal{C}}}

\newcommand{\cG}{{\mathcal{G}}}
\newcommand{\cK}{{\mathcal{K}}}

\newcommand{\dist}{\mathrm{dist}}

\newcommand{\Hol}{\mathrm{Hol}}

\newcommand{\beqa}{\begin{eqnarray*}}
\newcommand{\bea}{\begin{eqnarray}}
\newcommand{\eeqa}{\end{eqnarray*}}
\newcommand{\eea}{\end{eqnarray}}

\newtheorem{thm}{\bf Theorem}[section]
\newtheorem{coro}[thm]{\bf Corollary}

\newtheorem*{prop*}{\bf Proposition}

\newtheorem{rem}[thm]{\bf Remark}
\newtheorem{lem}[thm]{\bf Lemma}

\newcommand{\subeqsupset}{%
   \mathrel{\vcenter{
     \offinterlineskip
     \hbox{$\geqslant$}
  
     \hbox{$\leqslant$}
   }}%
}

\begin{document}
\title{Riesz bases of reproducing kernels in  small Fock spaces}
\author{K. Kellay \& Y. Omari}

\address{Address: Univ. Bordeaux, CNRS, Bordeaux INP, IMB, UMR 5251,   F--33405 Talence, France}  
\email{kkellay@math.u-bordeaux.fr}
\address{Laboratory of Analysis and Applications, Faculty of Sciences, Mohammed V University,  Rabat, Morocco.}
\email{omariysf@gmail.com}
\thanks{The research of the first  author is partially  supported by by the project ANR-18-CE40-0035 and by 
the Joint French-Russian Research Project PRC CNRS/RFBR,2017-2019\\
\hspace*{0.3cm} The second author is supported by CNRST Grant 84UM52016.
}

\date{}
\maketitle
{\begin{abstract} We give a complete characterization of Riesz bases of normalized reproducing kernels in the small Fock  spaces $\cF^2_{\varphi}$, { the spaces} of entire functions $f$ such that $f\mathrm{e}^{-\varphi} \in L^{2}(\C)$, where $\varphi(z)= (\log^+|z|)^{\beta+1}$, $0< \beta \leq 1$.
The first results in this direction are due to  Borichev-{Lyubarskii}  who showed that  $\varphi$ with $\beta=1$ is the largest weight for which the corresponding Fock space admits Riesz bases of reproducing kernels. {Later, such bases were characterized by Baranov-Dumont-Hartman-Kellay  in the case when $\beta=1$.}  The present paper answers a question in Baranov et al. by extending their results  for all parameters  $\beta\in (0,1)$.  Our results are analogous to those obtained for the case $\beta=1$  and those proved for Riesz bases of complex exponentials for the Paley-Wiener spaces.  We also obtain a  description of complete interpolating sequences in small Fock spaces with corresponding uniform norm. 
\end{abstract}}

\section{Introduction and statement of main results}

For a subharmonic function $\varphi$ tending to infinity when $|z|\to \infty$, define the weighted Fock spaces $\cF^{p}_\varphi$, $p=2,\infty$,   the spaces of all entire functions $f$ for which $f\mathrm{e}^{-\varphi}$ 
belongs to $L^p(\C)$. Seip and Wallst\'en \cite{S1,SW} characterized interpolating and sampling sequences in these spaces when $\varphi(z)=|z|^2$. They showed that there are no sequences which are simultaneously interpolating and sampling, and hence there are no  Riesz bases of reproducing kernels in this case. {However,} the situation changes in small Fock spaces when the weight increases slowly. In \cite[Theorem 2.8, Theorem 2.10]{BL}, Borichev and Lyubarskii proved that the spaces $\cF^2_\varphi$ with $\varphi(z)=\varphi_\beta(z)=\left(\log^+|z|\right)^{\beta+1}$, $0<\beta\leq 1$, possess Riesz bases of normalized reproducing kernels at real points. Another proof was given by Baranov, Belov and Borichev in \cite[Theorem 1.2]{BBB} for the case where $\varphi(z)=O(\log^+ |z|)^2$ with some regularity conditions. {In this case,} the Fock spaces $\cF^2_\varphi$ can be viewed as {\it de Branges spaces} \cite[Theorem 1.2]{BBB}, see also \cite{BBB1}. {A more general approach based on the boundedness and invertibility  of a discrete Hilbert transform} was  given  by Belov, {Mengestie} and Seip in \cite{BMS}.   In a recent work,  Baranov, Dumond, Hartmann and Kellay \cite{BDHK} gave a characterization of complete interpolating sequences (or Riesz bases) in $\cF^2_\varphi${, when $\varphi(z)=(\log^+|z|)^2$,} in the spirit of Ingham-{Kadets} 1/4 theorem for the Paley-Wiener space  (see \cite{O} for the case of $\cF^p_\varphi,\  1\leq p<\infty$).  If the weight grows more rapidly than $(\log^+ |z|)^2$, then the associated Fock type space has no Riesz bases of (normalized) reproducing kernels, see \cite{BBB,BDK,BL,S,SW} and the references therein. Also, for the case when the corresponding Riesz measure of the weight $\varphi$ is doubling, Marco, {Massaneda} and Ortega-Cerd{\`a} in \cite{MMO} proved that $\cF^2_\varphi$ possesses no Riesz bases of normalized reproducing kernels. {Finally,  it follows from \cite{IY} that there is no criterion for the existence 
 of
  Riesz bases in a Fock-type space $\cF^2_\varphi$ consisting of reproducing kernels in terms of the growth of the weight function $\varphi$ only.} \\

The main question we deal with in this paper was posed in \cite{BDHK}. One looks  for  a characterization of Riesz bases of normalized {reproducing} kernels for $\cF^2_{\varphi_\beta}$, where $0<\beta<1$.  Indeed, we obtain a complete description of such bases. Our results are analogous to the well known {Avdonin theorem} \cite{Av} established  for Riesz bases of complex exponentials in the Paley-Wiener spaces (see  \cite[Theorem 14, p. 178]{Y} and \cite[p. 102]{AI}  for more general statement) and to  those obtained for  $\cF^2_{\varphi_1}$ (see \cite[Theorem 1.1]{BDHK}). We  also give  a characterization of complete interpolating sequences in $\cF^{\infty}_{\varphi_\beta}$ for $0<\beta<1$. The case when $\beta=1$  was {studied} in \cite[Theorem 1.2]{BDHK}.

\subsection{Description of Riesz bases of $\cF^2_\varphi$}
 Let  $\varphi $  be  an increasing function defined on $[0,\infty)$  tending to infinity. We extend $\varphi$ into the whole complex plane $\C$  by $\varphi(z)=\varphi(|z|)$
 and  
we define the corresponding  {\it Fock type space} 
\[
\cF^2_\varphi:=\left\{f\in\Hol(\C)\ :\ \|f\|^{2}_{\varphi,2}:=\int_\C|f(z)|^2\mathrm{e}^{-2\varphi(z)}\mathrm{d}m(z)<\infty\right\},
\]
here $m$ stands for the area Lebesgue measure. { Since point evaluations are bounded linear functionals on}  $\cF^2_\varphi$,  by the representation Riesz Theorem the space $\cF^2_\varphi$ endowed with the  inner product  
\begin{eqnarray}
\langle f,g \rangle_\varphi\ :=\ \int_\C\ f(z)\overline{g(z)} \mathrm{e}^{-2\varphi(z)}\ \mathrm{d}m(z),\quad f,g\in\cF^2_\varphi, \nonumber
\end{eqnarray}
is a reproducing kernel Hilbert space (RKHS). Let $\mathrm{k}_z$ be its reproducing kernel at $z\in\C$, that is 
\[
f(z) = \langle f,\mathrm{k}_z  \rangle_\varphi, \qquad  f\in\cF^2_\varphi.
\]
We denote by $\Bbbk_z={\mathrm{k}_z }/{\|\mathrm{k}_z  \|_{\varphi,2}}$ the normalized reproducing kernel of $\cF^2_\varphi$. Let $\Gamma$ be a sequence of complex numbers. A system of normalized reproducing kernels $\cK_\Gamma=\left\{\Bbbk_\gamma\right\}_{\gamma\in\Gamma}$ is said to be a {\it Riesz basis} for $\cF^2_\varphi$ if $\cK_\Gamma$ is {\it complete} and if for some constant $C\geq 1$ we have 
$$
\frac{1}{C}\sum_{\gamma\in\Gamma }|a_\gamma|^2\leq \left\| \sum_{\gamma\in\Gamma }a_\gamma  \Bbbk_\gamma \right\|_{\varphi,2}^{2}\leq {C}\sum_{\gamma\in\Gamma }|a_\gamma|^2,
$$
 for each finite sequence $\{a_{\gamma}\}$. This means that $\Gamma$ is a {\it complete interpolating sequence} for $\cF^2_\varphi.$ In an equivalent way, the operator 
$$
\begin{array}{cccl}
T_\Gamma : & \cF^2_\varphi & \longrightarrow & \ell^2(\Gamma) \\
  & f & \longmapsto & \left(\langle f,\Bbbk_{\gamma}\rangle\right)_{\gamma\in\Gamma }
\end{array}$$
is bounded and invertible from $\cF^2_\varphi$ to $\ell^2(\Gamma)$.

{
We next set 
$$d(z,w)=\frac{|z-w|}{1+\min(|z|,|w|)},\qquad z,w\in \C.$$
A set $\Sigma\subset \C$ is said to be {\it $d-$separated} if there is 
$d_\Sigma>0$ such that
$$\inf \left\{d(\sigma,\sigma^*)\ :\  \sigma,\sigma^*\in \Sigma,\   \sigma\neq \sigma^*\right\}\ge d_\Sigma.$$ 
We denote by $\D(z,r)$ the Euclidean disk of radius $r$ centered at $z$. For $\delta<1$ and $0\neq z\in\C$ the ball 
corresponding to the distance $d$ is given by
$$ D_{d}(z,\delta):=\{w\in \C\ :\ d(z,w)<\delta\}.
$$
Hence,   $D_{d}(z,\delta)$ is comparable to  $\D(z,c_\delta |z|)$ with a suitable constant $c_\delta$ depending on $\delta$.  Thus  $\Sigma$ is $d-$separated if and only if there
exists $c>0$ such that the Euclidean disks $\D(\sigma, c|\sigma|)$,
$\sigma\in\Sigma$, are disjoint.} Our first main result in this paper is the following

\begin{thm}\label{thm30}
Let $\varphi(r)=\left(\log^+r\right)^{\beta+1}$ where $0< \beta < 1$, and let $\Gamma=\{\gamma_n\ :\ {n\geq 0}\}$ be a sequence of complex numbers such that $|\gamma_n|\leqslant |\gamma_{n+1}|$. We write $\gamma_n=\exp\Big(\frac{1+n}{1+\beta}\Big)^{\frac{1}{\beta}}\mathrm{e}^{\delta_n}\mathrm{e}^{i \theta_n}$, for every $n\geq 0$, where $(\delta_n)_n$ and $(\theta_n)$ are real sequences. Then $\cK_\Gamma=\{\Bbbk_\gamma\}_{\gamma\in\Gamma}$ is a Riesz basis  for $\cF^2_\varphi$ if and only if the following three conditions hold: 
\begin{enumerate}
\item\label{first} {$\Gamma$ is $d-$separated,}
\item\label{second} $\left(\delta_n/(1+n)^{\frac{1}{\beta}-1}\right)_{n\geq 0}$ is a bounded sequence,
\item\label{third} there exists $N\geqslant 1$ such that 
\begin{eqnarray}
\limsup_{n}\frac{1}{(1+n+N)^{\frac{1}{\beta}}-(1+n)^{\frac{1}{\beta}}}\Big|\sum_{k=n+1}^{n+N} 
\delta_k\Big|\ <  \frac{1}{2(1+\beta)^{\frac{1}{\beta}}}. \label{DeltaCondition}
\end{eqnarray}
\end{enumerate}
\end{thm}

\subsection{Complete interpolating sequences in  uniform Fock spaces $\cF^\infty_{\varphi}$.}

In this subsection we  consider  the uniform Fock type space, associated with $\varphi$, given as follows
\begin{eqnarray}
\cF^\infty_\varphi:=\left\{f\in\Hol(\C)\ :\ \|f\|_{\varphi,\infty}:=\sup_{z\in \C} |f(z)|\mathrm{e}^{-\varphi(z)} <\infty\right\}. \nonumber
\end{eqnarray}
$\cF^\infty_\varphi$ endowed with the norm $\|\ .\ \|_{\varphi,\infty}$ is clearly a Banach space. Let $\Gamma=\{\gamma_n\}$ be a sequence in $\C$. $\Gamma$ is said to be {\it sampling} for $\cF^\infty_\varphi$ whenever there exists $C\geq 1$ such that 
\begin{eqnarray*}
\|f\|_{\varphi,\infty}\ \leq\ C \|f\|_{\varphi,\infty,\Gamma},
\end{eqnarray*}  
for every $f\in\cF^\infty_\varphi$, where
\[\|f\|_{\varphi,\infty,\Gamma} :=\ \sup_{\gamma\in \Gamma} |f(\gamma)|\mathrm{e}^{-\varphi(\gamma)}.\]
$\Gamma$ is called an {\it interpolating set} for $\cF^\infty_\varphi$ if for every sequence $v=\left(v_\gamma\right)_{\gamma\in\Gamma}$ such that $\|v\|_{\varphi,\infty,\Gamma}<\infty$ there exists $f\in\cF^\infty_\varphi$ that satisfies $f(\gamma) = v_\gamma$, for every $\gamma\in\Gamma$. Finally, $\Gamma$ is said to be  a {\it complete interpolating sequence} for $\cF^\infty_\varphi$ if $\Gamma$ is simultaneously a sampling and an interpolating set for $\cF^\infty_\varphi$. \\

The following result describes complete interpolating sequences for $\cF^\infty_\varphi$,  the case  $\beta=1$  was obtained in \cite{BDHK}.

\begin{thm}\label{thmInfty}  Let $\varphi(r)=\left(\log^+r\right)^{\beta+1}$ where $0< \beta < 1$. {Let $\Gamma$  be a sequence of complex numbers} and let $\gamma^*\in\C\setminus\Gamma$. Then $\Gamma\cup\{\gamma^*\}$ is a complete interpolating set for $\cF^\infty_\varphi$ if and only if  $\cK_\Gamma$ is a Riesz basis for $\cF^2_\varphi$.  
\end{thm}

\subsection{Remarks} 
Several remarks are in order before we turn to the proofs of our theorems:\\

{{\bf 1.} Complete interpolating sequences  for $\cF_\varphi^p$, $p=2,\infty$   are characterized by comparing them to the  sequence
 \begin{equation}\label{Lambda}
 \Lambda=\left\{\lambda_n:=\exp{\Big(\frac{1+n}{1+\beta}\Big)^{\frac{1}{\beta}}}\mathrm{e}^{i \theta_n},\quad \theta_n\in \R \text{ : } n\geq 0\right\},
 \end{equation}
  which is a complete interpolating sequence  for $\cF^2_\varphi$ and $\Lambda\cup \{\lambda^*\}$ is a complete interpolating sequence  for $\cF^\infty_\varphi$ where $\lambda^*\in \C\setminus\Lambda$ (for $p=2$ see \cite[Theorem 1.2]{BBB} and Lemma \ref{thm1} and for the case $p=\infty$ see  Lemma \ref{propo} ). }    \\

{{\bf 2.}  The diagonal asymptotic estimates of the reproducing kernel {play} an important role  in the study of Riesz bases.  Let $\rho(z):=\left(\Delta\varphi(z)\right)^{-1/2}=|z|\left(\log|z|\right)^{\frac{1-\beta}{2}}$, {as} stated in Corollary \ref{keernel} the kernel admits the following estimates
$$
\frac{\mathrm{e}^{2\varphi(z)}}{\rho(z)^2}\ \lesssim \left\|\mathrm{k}_z\right\|_{\varphi,2}^2\ \lesssim\ \frac{\mathrm{e}^{2\varphi(z)}}{|z|\rho(z)},\qquad |z|>1.
$$
Furthermore the upper estimate is attained on a subset of $z$ (see Lemma \ref{Kernel_Estimate} for the precise diagonal asymptotic estimates)}.\\

{\bf 3. }The case $\beta = 1$ was treated in \cite[Theorem 1.1]{BDHK}  {under} conditions \eqref{first}, \eqref{second} and  under the hypothesis that  there exists $N\geqslant 1$ such that 
\begin{eqnarray}
\sup_{n}\frac{1}{N}\Big|\sum_{k=n+1}^{n+N} 
\delta_k\Big| <  \frac{1}{4}.
 \label{DeltaBDHK}
\end{eqnarray}
The family $\cK_\Lambda=\{\Bbbk_\lambda\}_{\lambda\in\Lambda}$ (associated with $\beta=1$) is a Riesz basis  for $\cF^2_{\varphi_1}$. {Condition \eqref{DeltaCondition} may appear different to \eqref{DeltaBDHK} in the case $\beta=1$; however, these two conditions are equivalent for the problem studied here (see Remark \ref{rem1}).} \\

{\bf 4.} Condition \eqref{DeltaCondition} can be written as
\begin{eqnarray}\label{Deltacondition2}
\Delta_N:=\limsup_{n}\frac{1}{\log|\lambda_{n+N}|-\log|\lambda_n|}\Big|\sum_{k=n+1}^{n+N} 
\delta_k\Big|\ < \frac{1}{2}.
\end{eqnarray}
For $N=1$,  condition \eqref{DeltaCondition} is equivalent to
\begin{equation}\label{kadetscond}
\underset{n\geq 1}{\sup}\ \frac{ \big|\log |\gamma_n|-\log|\lambda_n|\big|}{n^{\frac{1}{\beta}-1}}=\underset{n\geq 1}{\sup}\ \frac{
|\delta_n|}{n^{\frac{1}{\beta}-1}} <\frac{1}{2\beta(1+\beta)^{\frac{1}{\beta}}},
\end{equation}
(see Remark \ref{rem1}).  { Conditions \eqref{Deltacondition2} and  \eqref{kadetscond} look  similar to those used in some related results}.  We mention here  the  well known 1/4 {Kadets-Ingham's theorem} \cite{Ka} on  Riesz bases of exponentials in the Paley-Wiener spaces $PW^2_\alpha$, and the results  by Marzo and Seip  \cite{MS} for  spaces of polynomials. For fixed $N$, \eqref{Deltacondition2}  looks similar to Avdonin's condition \cite{Av}. {It is interesting to note that  conditions \eqref{first}-\eqref{third} describe completely Riesz bases of reproducing kernels in $\cF^2_\varphi$; on the contrary, in the Paley-Wiener spaces, these conditions are just sufficient}.\\

{\bf 5.} We notice again that our spaces possess Riesz bases of normalized reproducing kernels at real points, and hence they can be viewed as de Branges spaces \cite{BBB1, BBB}. Using the boundedness and invertibility results on the discrete Hilbert transform on lacunary sequences,   Belov, Mengestie and  Seip  {gave} another characterization of Riesz bases, where our  summability condition  \eqref{DeltaCondition} corresponds  to a Muckenhoupt-type condition;   for further details we refer to \cite[Theorem 1.1]{BMS}.  Our approach consists of using  Bari's theorem on  Riesz bases in Hilbert spaces,  as in  the proof of \cite[Theorem 1.1]{BDHK}.\\

{The plan of our paper is as following. In the next section we state diagonal asymptotic estimates on the reproducing kernel. We then show  that $\Lambda$ and $\Lambda\cup\{\lambda^*\}$ are complete interpolating sequences for $\cF^2_\varphi$ and $\cF^\infty_\varphi$, respectively. {Furthermore, we}  deal with the separation condition. Section \ref{sectthree} is devoted to proving Theorem \ref{thm30}. The proof of Theorem \ref{thmInfty} is presented in Section \ref{sectioninftinity}. We end our paper by some remarks in the last section.}\\

 Throughout the paper, we use the following notations:  
 \begin{itemize}
 \item $A\lesssim B$ means that there is an absolute constant $C$ such that $A \leq CB$.  
 \item $A\asymp B$  if both $A\lesssim B$ and $B\lesssim A$ hold. 
 \end{itemize}
 
\subsection*{Acknowledgement.} The authors are grateful to A. Borichev, O. El-Fallah,  Yu. Lyubarskii and J. Ortega-Cerd\`a for helpful discussions.  We also would like to thank the referees for their detailed comments and suggestions.

\section{Key lemmas and preliminary results}

In this section we prove some preliminary results and some key lemmas. First, we establish estimates on the reproducing kernel at the diagonal  also get necessary estimates on certain infinite products and also  preliminary results concerning the separation condition.  Throughout the rest of this paper we denote $\varphi(r)=\varphi_\beta(r)=\left(\log^+r\right)^{\beta+1}$, where $0<\beta<1$. 

\subsection{Estimates on the norm of the reproducing kernel}
We begin by the following lemma which was proved in \cite[Lemma 3.1]{BBB} and which will
be used repeatedly throughout the paper.

\begin{lem}\label{moment}
For every nonnegative integer $n$ we have
\[
\|z^n\|^{2}_{\varphi,2} \asymp \Big(\frac{1+n}{1+\beta}\Big)^{\frac{1-\beta}{2\beta}}\exp\Big(2\beta\Big(\frac{1+n}{1+\beta}\Big)^{\frac{1+\beta}{\beta}}\Big).
\]
\end{lem}

Next we produce an asymptotic estimate on the reproducing kernel $\mathrm{k}_z$ of $\cF_\varphi^2$  on the diagonal. Note that such estimate was obtained for the points $\lambda\in\Lambda$ by Baranov, Belov and Borichev in \cite[Lemma 3.2]{BBB}. They proved  the  estimate directly by using Hardy's convexity theorem. Here we use Laplace's method to deal with the rest of complex numbers $z\in\C$. {We denote by $[x]$ the integer part of a real $x$}.

\begin{lem}\label{Kernel_Estimate} Let   $|z|=\mathrm{e}^s$, $n_z=[(1+\beta)s^\beta]$. Set $g_s(t)=st-\beta\big(\frac{t}{1+\beta}\big)^{\frac{1+\beta}{\beta}}$. Let 
$\widetilde{\varphi}(z)=\max( g_{s}(n_z), g_{s}(n_z+1))$. We have 
$$\|\mathrm{k}_z\|^{2}_{\varphi,2}\asymp  \frac{ \mathrm{e}^{2\widetilde{\varphi}(z)}}{|z|\rho(z)}+\frac{ \mathrm{e}^{2\varphi(z)}}{\rho^2(z)},\qquad |z|>1,$$
where $\rho(z):=\left(\Delta\varphi(z)\right)^{-1/2}=|z|\left(\log|z|\right)^{\frac{1-\beta}{2}}.$ 
\end{lem}
\begin{proof} Let  $|z|=\mathrm{e}^{s}$.  Note that $g'_s(t)=0$ for  $t=(1+\beta) s^\beta$ and $\widetilde{\varphi}(z)\leq \varphi(z)$. 
Since  $\|\mathrm{k}_z\|^{2}_{\varphi,2}=\sum_{n\geq 0} {{|z|^{2n}}}/{\|z^n\|^{2}_{\varphi,2}}$, by Lemma \ref{moment} we have
  \begin{multline*}
 \mathrm{e}^{2s} \|\mathrm{k}_{\mathrm{e}^s}\|^{2}_{\varphi,2}\asymp 
 \sum_{n\geq 1}\frac{\mathrm{e}^{2g_s(n)}}{ (1+n)^{\frac{1-\beta}{2\beta}}}\\
 \lesssim \int_0^{n_z}
 \frac{\mathrm{e}^{2g_s(t)}}{(1+t)^{\frac{1-\beta}{2\beta}}}\mathrm{d}t
+ \frac{\mathrm{e}^{2g_s(n_z)}}{ (1+n_z)^{\frac{1-\beta}{2\beta}}} + \frac{\mathrm{e}^{2g_s(n_z+1)}}{ (2+n_z)^{\frac{1-\beta}{2\beta}}}+ \int_{n_z+1}^{\infty}
 \frac{\mathrm{e}^{2g_s(t)}}{(1+t)^{\frac{1-\beta}{2\beta}}}\mathrm{d}t\\
 \lesssim   \frac{\mathrm{e}^{2\widetilde{\varphi}(z)}}{s^{\frac{1-\beta}{2}}}+  \int_0^\infty
 \frac{\mathrm{e}^{2s^{1+\beta}\big[ x - \beta\big(\frac{x}{1+\beta}\big)^{\frac{1+\beta}{\beta}} \big]}}{ s^\frac{1-\beta}{2}(1+x)^{\frac{1-\beta}{2\beta}}} s^\beta \mathrm{d}x:=\frac{\mathrm{e}^{2\widetilde{\varphi}(z)}}{s^{\frac{1-\beta}{2}}}+ \frac{s^\beta}{s^\frac{1-\beta}{2}}\int_0^\infty \mathrm{e}^{-2s^{1+\beta} \psi(x)}f(x)\mathrm{d}x,
\end{multline*}
where   $\psi(x)= \beta\big(\frac{x}{1+\beta}\big)^{\frac{1+\beta}{\beta}}-x$ and { $f(x)=1/(1+x)^{\frac{1-\beta}{2\beta}}$}. We have 
 $\psi'(x) = \big(\frac{x}{1+\beta}\big)^{\frac{1}{\beta}}-1 .$
Let $x_0=1+\beta$. Then    $\psi'(x_0)=0$ and  $\psi(x_0)=-1$. Since   
$\psi''(x)=\frac{x^{\frac{1}{\beta}-1}}{\beta(1+\beta)^{\frac{1}{\beta}}}>0$, $x\geq 1$, we get by 
 Laplace's method  \cite{E} 
 \begin{equation*}\label{Laplace}
 \int_0^\infty \mathrm{e}^{-S\psi(x)}f(x)\mathrm{d}x\sim \sqrt{\frac{2\pi}{S\psi''(x_0)}}f(x_0)e^{-S\psi(x_0)},\qquad S\to +\infty.
 \end{equation*}
For $S=2s^{1+\beta}\to +\infty$,  we get 
  $$ \|\mathrm{k}_{z}\|^{2}_{\varphi,2}= \|\mathrm{k}_{\mathrm{e}^s}\|^{2}_{\varphi,2}\lesssim \frac{ \mathrm{e}^{2\widetilde{\varphi}(z)}}{\mathrm{e}^{2s} s^{\frac{1-\beta}{2}}}+ \frac{\mathrm{e}^{2s^{1+\beta}}}{\mathrm{e}^{2s} s^{1-\beta}}=\frac{ \mathrm{e}^{2\widetilde{\varphi}(z)}}{|z|^2(\log  |z|)^{\frac{1-\beta}{2}}}+ \frac{\mathrm{e}^{2\varphi(z)}}{|z|^2 (\log|z|)^{1-\beta}}.$$
  On the other hand, using Laplace's method again, we obtain
 \begin{eqnarray*}
{\mathcal{I}}:=  \int_{0}^{1+\beta}
 {\mathrm{e}^{2s^{1+\beta}\big[ x - \beta\big(\frac{x}{1+\beta}\big)^{\frac{1+\beta}{\beta}} \big]}} \mathrm{d}x &\sim&\sqrt{\frac{2\pi}{2s^{1+\beta}\psi''(x_0)}}\ \mathrm{e}^{2s^{1+\beta}}\\
 &=&  \sqrt{\frac{\pi\beta(1+\beta)}{s^{1+\beta}}}\ \mathrm{e}^{2s^{1+\beta}}.
 \end{eqnarray*}
Also, we have 
\begin{eqnarray*}
\mathcal{J}&:=&\int^{1+\beta}_{1+\beta-2/s^\beta}
{\mathrm{e}^{2s^{1+\beta}\big[ x - \beta\big(\frac{x}{1+\beta}\big)^{\frac{1+\beta}{\beta}} \big]}} \mathrm{d}x\\
 &=&\int_0^{2/s^\beta}\exp \Big({2s^{1+\beta}\big[ (1+\beta-u) - \beta\big(1-\frac{u}{1+\beta}\big)^{\frac{1+\beta}{\beta}} \big]}\Big) \mathrm{d}u\\
& = & \int_0^{2/s^\beta}\exp \Big(2s^{1+\beta}\big[ 1-\frac{1}{2\beta(\beta+1)} u^2\big]\Big) \mathrm{d}u \\ 
& \leq & \mathrm{e}^{2s^{1+\beta}}\int_0^{\infty}\ \mathrm{e}^{-\frac{s^{1+\beta} u^2}{\beta(\beta+1)} } \mathrm{d}u \\
 & = & \  \frac{1}{2}\sqrt{\frac{\beta(\beta+1)\pi}{ s^{1+\beta}}}\ \mathrm{e}^{2s^{1+\beta}}\ \sim\ \frac{1}{2} \mathcal{I}.
\end{eqnarray*}
So, we get 
 \begin{eqnarray*}
 \mathrm{e}^{2s} \|\mathrm{k}_{\mathrm{e}^s}\|^{2}_{\varphi,2} 
& \gtrsim &\int_0^{n_z-1}
 \frac{\mathrm{e}^{2g_s(t)}}{(1+t)^{\frac{1-\beta}{2\beta}}}\mathrm{d}t+  \frac{\mathrm{e}^{2g_s(n_z)}}{ (1+n_z)^{\frac{1-\beta}{2\beta}}} + \frac{\mathrm{e}^{2g_s(n_z+1)}}{ (2+n_z)^{\frac{1-\beta}{2\beta}}}\\
 & \gtrsim &  \frac{ \mathrm{e}^{2\widetilde{\varphi}(z)}}{s^{\frac{1-\beta}{2}}}+  \frac{s^\beta}{s^{\frac{1-\beta}{2}}} \int_{0}^{1+\beta-{2}/{s^\beta}}
 \frac{\mathrm{e}^{2s^{1+\beta}\big[ x - \beta\big(\frac{x}{1+\beta}\big)^{\frac{1+\beta}{\beta}} \big]}}{ (1+x)^{\frac{1-\beta}{2\beta}}} \mathrm{d}x  \\
  & = &  \frac{ \mathrm{e}^{2\widetilde{\varphi}(z)}}{s^{\frac{1-\beta}{2}}}+  \frac{s^\beta}{s^{\frac{1-\beta}{2}}} \big(\mathcal{I}-\mathcal{J}\big)  \\
 &\asymp&\frac{ \mathrm{e}^{2\widetilde{\varphi}(z)}}{s^{\frac{1-\beta}{2}}}+   \frac{s^\beta}{s^{\frac{1-\beta}{2}}} \frac{\mathrm{e}^{2s^{1+\beta}}}{\sqrt{s^{1+\beta}}}.
  \end{eqnarray*}
Therefore, 
$$\|\mathrm{k}_z\|_{\varphi,2}^{2} \gtrsim \frac{ \mathrm{e}^{2\widetilde{\varphi}(z)}}{|z|^2(\log  |z|)^{\frac{1-\beta}{2}}}+ \frac{\mathrm{e}^{2\varphi(z)}}{|z|^2 (\log |z|)^{1-\beta}},\qquad |z|>1.$$
 This completes the proof.
   
\end{proof}

We have the following corollary, with identity \eqref{kernelBBB}  obtained already in \cite[Lemma 3.2]{BBB}
\begin{coro}\label{keernel}
We have 

\begin{eqnarray}\label{kernEstim}
\frac{\mathrm{e}^{2\varphi(z)}}{\rho(z)^2}\ \lesssim \left\|\mathrm{k}_z\right\|_{\varphi,2}^2\ \lesssim\ \frac{\mathrm{e}^{2\varphi(z)}}{|z|\rho(z)},\qquad |z|>1.
\end{eqnarray}
Furthermore, for every $\lambda_n\in\Lambda$, where $\Lambda$ is the sequence given in \eqref{Lambda}, we have
\begin{equation}\label{kernelBBB}
\left\|\mathrm{k}_{\lambda_n}\right\|_{\varphi,2}^2 \ \asymp\ \frac{\mathrm{e}^{2\varphi(\lambda_n)}}{|\lambda_n|\rho(\lambda_n)}\asymp |e_n(\lambda_n)|^2,
\end{equation}
where $e_n(z)=z^n/\|z^n\|_{\varphi,2}$. {Also, if $\sigma_n=\exp\left(\frac{n+1/2}{1+\beta}\right)^{\frac{1}{\beta}},\ n\geq 2$, then 
\begin{eqnarray}\label{kernel}
\left\|\mathrm{k}_{\sigma_n}\right\|_{\varphi,2}^2 \ \asymp\ \frac{\mathrm{e}^{2\varphi(\sigma_n)}}{\rho(\sigma_n)^2}.
\end{eqnarray}
}
\end{coro}
{
\begin{proof} Estimate \eqref{kernEstim} follows directly from the previous lemma. \eqref{kernelBBB} was proved in \cite[Lemma 3.2]{BBB}. We will prove \eqref{kernel}. Following the notations of Lemma \ref{Kernel_Estimate}, let $|z|=\mathrm{e}^s$ and set $n_z:=[(1+\beta)s^\beta]$, $g_s(t)=st-\beta\left(\frac{t}{1+\beta}\right)^{\frac{\beta+1}{\beta}}$. By Lemma \ref{Kernel_Estimate} we have
\begin{eqnarray*}
\|\mathrm{k}_z\|_{\varphi,2}^2\ \asymp\ \frac{1}{|z|^2(\log |z|)^{\frac{1-\beta}{2}}}\left[\mathrm{e}^{2\max\{g_s(n_z),g_s(n_z+1)\}}+\mathrm{e}^{2\varphi(z)-2\frac{1-\beta}{4}\log\log|z|}\right].
\end{eqnarray*}
We write $n_z=(1+\beta)s^\beta+\delta_z$, for some $\delta_z\in(-1,0]$. Simple calculations show that for $j\in\{0,1\}$ we have
\begin{multline*}
g_s(n_z+j)-\varphi(z)+\frac{1-\beta}{4}\log s  =  s\left((1+\beta)s^\beta+\delta_z+j\right)\\
-\beta\left(\frac{(1+\beta)s^\beta+\delta_z+j}{1+\beta}\right)^{\frac{\beta+1}{\beta}} -s^{\beta+1}+\frac{1-\beta}{4}\log s\nonumber\\
 =  \frac{1-\beta}{4}\log s-\frac{1+o(1)}{2\beta(1+\beta)}(\delta_z+j)^2 s^{1-\beta} \underset{|z|\rightarrow\infty}{\longrightarrow} -\infty, \nonumber
\end{multline*}
if and only if $$\frac{\beta(1-\beta^2)}{2+o(1)}\ \frac{\log s}{s^{1-\beta}}{\ =}\ o\left((\delta_z+j)^2\right), \qquad |z|\to\infty.$$
Now, if $z=\sigma_n$, then  $\delta_z^2=(\delta_z+1)^2=\frac{1}{4}$ and hence the latter estimate holds. Thus \eqref{kernel} follows immediately.
\end{proof}
}
\subsection*{Remark} For regular radial weights  satisfying the estimate  $(\log r)^2 = O(\varphi(r))$, {when} $r\to \infty$, the reproducing kernel of $\cF^2_\varphi$ possesses the property  $k_z(z) \asymp \Delta \varphi(z)\mathrm{e}^{2\varphi(z)}$ (see \cite{BBB, BDK, BL}).  {This estimate remains valid for the kernel of $\cF^2_\varphi$ when the associated Riesz measure of the weight $\varphi$ is doubling \cite[Lemma 21]{MMO}}. Estimate  \eqref{kernelBBB} proves that the case $\beta< 1$ is quite different from the case when $\beta=1$. Indeed, for $\beta=1$ the reproducing kernel admits the estimate $\left\|\mathrm{k}_z\right\|^2_{\varphi_1,2}\asymp  \Delta\varphi_1(z) \mathrm{e}^{2\varphi_1(z)}$, for every $z\in\C.$ However, in \eqref{kernelBBB} we have for every $\lambda\in\Lambda$ that  $\left\|\mathrm{k}_{\lambda}\right\|_{\varphi_\beta,2}^2 \asymp \mathrm{e}^{2\varphi_\beta(\lambda)}/\Big(|\lambda|^2\Big(\log(\lambda)\Big)^{\frac{1-\beta}{2}}\Big)$, and hence $ \Delta\varphi_\beta(\lambda) \mathrm{e}^{2\varphi_\beta(\lambda)}=o\left(\|\mathrm{k}_{\lambda}\|_{\varphi_\beta,2}^2\right)$, whenever $0<\beta<1$. This could  explain the difference in the results obtained about Riesz bases of normalized reproducing kernels in the situation $\beta=1$ and $0<\beta<1$.\\

\subsection{Complete interpolating sequences in  $\cF^p_{\varphi}$, $p=2,\infty$}

With a given sequence $\Gamma$ of complex numbers, we associate  the infinite product
\[
G_\Gamma(z)=\prod_{\gamma\in\Gamma}\left(1-\frac{z}{\gamma}\right) ,\qquad z\in\C,
\]
whenever it converges. In what follows, we denote by $\dist(z,\Gamma)$ the Euclidean distance between $z$ and the sequence $\Gamma$.

 The following lemma provides estimates on the infinite product $G_\Gamma$ associated with a sequence $\Gamma$ defined by
 \begin{equation}
 \Gamma:=\left\{\gamma_n:=\exp\Big(\frac{1+n}{1+\beta}\Big)^{\frac{1}{\beta}}
 \mathrm{e}^{\delta_n}\mathrm{e}^{i \theta_n},\quad \theta_n\in \R\ :\ n\geqslant 0 \right\}.\label{Gamm}
 \end{equation}
 Recall that 
 $$\Delta_N:=\limsup_{n\to \infty}\frac{1}{\log|\lambda_{n+N}|-\log|\lambda_n|}\Big|\sum_{k=n+1}^{n+N} 
\delta_k\Big|.$$

\begin{lem}\label{lem-estim}
 Let $\Gamma=\{\gamma_n\}_n$  be a sequence defined in \eqref{Gamm}. We write $\gamma_n=\lambda_n\mathrm{e}^{\delta_n}\mathrm{e}^{i\theta_n}$, where $\delta_n,\theta_n\in\R$, and suppose that $|\gamma_n|\leq |\gamma_{n+1}|$ and $|\delta_n|=O\big((1+n)^{\frac{1}{\beta}-1}\big)$. The infinite product $G_\Gamma$ converges on every compact set of $\C$ and for every small $\varepsilon>0$ there exists a positive constant $C$ such that 
\[
\frac{1}{C}\frac{\dist(z,\Gamma)}{(1+|z|)^{\frac{3}{2}+\Delta_N +\varepsilon}}\ \leq \left|G_\Gamma(z)\right|  \mathrm{e}^{-\varphi(z)}\leq C\frac{\dist(z,\Gamma)}{(1+|z|)^{\frac{3}{2}-\Delta_N -\varepsilon}},\ \qquad z\in\C.
\]
\end{lem}
\begin{proof}
Let $z\in\C$ with $|z|=\mathrm{e}^t=\exp{\big(\frac{s}{1+\beta}\big)^{\frac{1}{\beta}}}$. Let also $m$ be the integer such that $|\gamma_{m-1}|\leq |z|< |\gamma_m|$ and suppose that $\dist(z,\Gamma)=|z-\gamma_{m-1}|$.
Then 
\begin{eqnarray*}
\log|G_\Gamma(z)| & = &\sum_{0\leqslant n<m-1}\log\left|1-\frac{z}{\gamma_n}\right| + \log\left|1-\frac{z}{\gamma_{m-1}}\right|+ O(1) \\
       & = & \sum_{n=0}^{m-1}\log\left|\frac{z}{\gamma_n}\right| + \log \dist(z,\Gamma) - t + O(1)  \\
       & =&  mt-\sum_{n=1}^{m}\left(\frac{k}{1+\beta}\right)^{\frac{1}{\beta}} -\sum_{n=1}^{m-1}\delta_k + \log \dist(z,\Gamma) - t + O(1),  \quad t\to \infty.
 \end{eqnarray*}
Furthermore, for every $\alpha\neq -1$ we have
\begin{equation}\label{mac}
\sum_{k=1}^{m}k^\alpha  =  \frac{m^{\alpha+1}}{\alpha+1} + \frac{m^\alpha}{2}\left(1+o(1)\right).
\end{equation}
 
Now for every $N\in\N$ there exist  two positive integers $p$ and $r$ such that $m=pN+r$ and $0\leqslant r<N$. Set $u_k := \log |\lambda_k|$. We have
\begin{eqnarray}
\left|\sum_{k=1}^{m}\ \delta_k\right| & \leqslant & \sum_{j=1}^{p-1}\left|\sum_{k=1}^{N}\ \delta_{jN+k}\ \right| + \left|\sum_{k=pN+1}^{pN+r}\delta_k \right| \nonumber \\
  & \leqslant & \left(\Delta_N +\frac{\varepsilon}{2}\right) \sum_{j=0}^{p-1} \left(u_{jN+N}-u_{jN}\right) + {O(m^{\frac{1}{\beta}-1}) }\nonumber \\
  & \leqslant & \left(\Delta_N+\frac{\varepsilon}{2}\right) t + o(t), \quad t\to \infty. \nonumber
\end{eqnarray}
Consequently,
{\small
\begin{eqnarray}
\log\left|G_\Gamma(z)\right| &
\subeqsupset & mt -\beta \left(\frac{m}{1+\beta}\right)^{\frac{1}{\beta}+1} - \left(\frac{3}{2}\mp \Delta_N \mp \frac{\varepsilon}{2}\right)\left(\frac{m}{1+\beta}\right)^{\frac{1}{\beta}} + \log \dist(z,\Gamma) + o(t) \nonumber\\
 & \subeqsupset & t^{\beta+1} -\left(\frac{3}{2}\mp \Delta_N \mp \frac{\varepsilon}{2} \right)t + \log \dist(z,\Gamma) + o(t). \nonumber
\end{eqnarray}
}
Finally, there exist  two positive numbers $A$ and $B$  depending on $\varepsilon$ such that
\[
A \frac{\dist(z,\Lambda)}{(1+|z|)^{\frac{3}{2}+\Delta_N +\varepsilon}}\ \leq\ \left|G_\Gamma(z)\right|\  \mathrm{e}^{-\varphi(z)} \leq\ B\frac{\dist(z,\Lambda)}{(1+|z|)^{\frac{3}{2}-\Delta_N -\varepsilon}}\ , \qquad z\in \C.
\]
\end{proof}
We need the following standard ingredient,  which we single out as a lemma. 
\begin{lem}\label{integ}  Let $\Gamma$  be a sequence defined as  in  \eqref{Gamm}. We have
 \[
\int_\C \frac{\dist(z,\Gamma)^2}{(1+|z|)^{3+\alpha}}\ \mathrm{d}m(z)<\infty \iff \alpha> 1.
\]
\end{lem}
\begin{proof}Indeed, for every $\frac{1}{2}(|\gamma_{n-1}|+|\gamma_{n}|)\leqslant r\leqslant \frac{1}{2}(|\gamma_{n}|+|\gamma_{n+1}|)$, we have $\dist(r\mathrm{e}^{i\theta},\Gamma)\asymp|r\mathrm{e}^{i\theta}-\gamma_n|$. So, we get
$$
\int_\C \frac{\dist(z,\Gamma)^2}{(1+|z|)^{3+\alpha}} \mathrm{d}m(z)  \asymp \sum_{n\geqslant 0} \int_{\frac{1}{2}(|\gamma_{n-1}|+|\gamma_{n}|)}^{\frac{1}{2}(|\gamma_{n}|+|\gamma_{n+1}|)}\int_{0}^{2\pi} \frac{|r\mathrm{e}^{i\theta}-\gamma_n|^2}{(1+r)^{3+\alpha}} 
r\mathrm{d}\theta \mathrm{d}r 
       \asymp  \sum_{n\geqslant 0}\frac{1}{ |\gamma_{n}|^{\alpha-1}}.
$$
Therefore, the last sum is finite if and only if $\alpha>1$.

\end{proof}
Next we extend a result of 
Borichev and Lyubarskii  \cite[Theorem 2.8]{BL}. They  proved  that $\{\Bbbk_{\widetilde{\lambda}_n}\}_{n\geq 0}$ is a Riesz { basis}  for $\cF^2_\varphi$ where $\widetilde{\lambda}_0=0$, $|\widetilde{\lambda}_n|= \exp \left[(w_{n+1}-w_{n-1})/4\right]$ and 
$$w_n:=\log\|z^n\|^2=c(n+1)^{1+\frac{1}{\beta}}+O(\log n).$$
Their proof was based on  Bari's theorem (see \cite[section A.5.7.1]{N1}). Using  the estimate on the moments $\|z^n\|_{\varphi,2}$ given in Lemma \ref{moment}, we produce  a family of sequences of similar  kind. These sequences could be viewed as a reference family of complete interpolating sequences for $\cF^2_\varphi$ and will play a crucial tool in the proof of the main theorems.  A proof of this lemma was given by Baranov-Belov-Borichev in \cite[Theorem 1.2]{BBB}  using the results of Belov-Mengestie-Seip  \cite{BMS}. However, for the sake of completeness, we give here another proof that uses Bari's theorem as in \cite{BL}.
\begin{lem}\label{thm1}  Let $\Lambda$  be a sequence defined by  \eqref{Lambda}. Then 
$\cK_\Lambda=\left\{\Bbbk_{\lambda}\right\}_{ \lambda\in\Lambda}$ is a Riesz basis  for $\cF^2_\varphi$.
\end{lem}

\begin{proof}
 Let $n\geqslant 0$, $e_n(z)=z^n/\|z^n\|_{\varphi,2}$, and $h_n:=\mathrm{e}^{-in\theta_n}e_n$. We have 
\begin{eqnarray*}
\left\|h_n-\Bbbk_{\lambda_n}\right\|^{2}_{\varphi,2}& = & \left\|\mathrm{e}^{-in\theta_n}e_n-\frac{\overline{e_n(\lambda_n)}}{\|\mathrm{k}_{\lambda_n}\|_{\varphi,2}}e_n-\sum_{k\neq n} \frac{\overline{e_k(\lambda_n)}}{\|\mathrm{k}_{\lambda_n}\|_{\varphi,2}}e_k\right\|^{2}_{\varphi,2} \\
&=& \underset{J_1}{\underbrace{\left|1-\frac{|e_n(|\lambda_n|)|}{\|\mathrm{k}_{\lambda_n}\|_{\varphi,2}}\right|^2}}+ \underset{J_2}{\underbrace{\sum_{k\neq n} \frac{|e_k(\lambda_n)|^2}{\|\mathrm{k}_{\lambda_n}\|^2_{\varphi,2}}}}. \nonumber
\end{eqnarray*}
Furthermore, 
$$J_1 = \left|1-\frac{|e_n(|\lambda_n|)|}{\|\mathrm{k}_{\lambda_n}\|_{\varphi,2}}\right|^2\leq   1-\frac{|e_n(\lambda_n)|^2}{\|\mathrm{k}_{\lambda_n}\|^2_{\varphi,2}} = J_2.$$
Consequently,
\begin{eqnarray}
\left\|h_n-\Bbbk_{\lambda_n}\right\|^2_{\varphi,2} \ \asymp\  \sum_{k\neq n} \frac{|e_k(\lambda_n)|^2}{\|\mathrm{k}_{\lambda_n}\|^2_{\varphi,2}} \
 \lesssim \ \sum_{k\neq n} \left|\frac{e_k(\lambda_n)}{e_n(\lambda_n)}\right|^2 .\nonumber
\end{eqnarray}
On the other hand, for $k\neq n$ we have
\begin{eqnarray}
\frac{|e_n(\lambda_n)|^2}{|e_k(\lambda_n)|^2} & \asymp & \left(\frac{1+k}{1+n}\right)^{\frac{1-\beta}{2\beta}}\mathrm{e}^{2c(n,k)}, \nonumber
\end{eqnarray}
where
\begin{eqnarray}
c(n,k) & := & (n-k)  \left(\frac{1+n}{1+\beta}\right)^{\frac{1}{\beta}}+\beta \left[   \left(\frac{1+k}{1+\beta}\right)^{\frac{1+\beta}{\beta}} - \left(\frac{1+n}{1+\beta}\right)^{\frac{1+\beta}{\beta}}  \right]. \nonumber
\end{eqnarray}
Therefore, by simple computations we get
$$c(n,k) \geqslant (1+\beta)^{-\frac{1+\beta}{\beta}} |k-n|^2(1+n)^{\frac{1}{\beta}-1}.
$$
Thus,
\begin{eqnarray}
\sum_{n=0}^{\infty}\left\|h_n-\Bbbk_{\lambda_n}\right\|^2_{\varphi,2} & \asymp & \sum_{n=0}^\infty\sum_{k\neq n} \frac{|e_k(\lambda_n)|^2}{\|\mathrm{k}_{\lambda_n}\|^2_{\varphi,2}} \nonumber\\
    & \lesssim & \sum_{n=0}^\infty\sum_{k\neq n} \left(\frac{1+n}{1+k}\right)^{\frac{1-\beta}{2\beta}}\ \mathrm{e}^{-2  (1+\beta)^{-\frac{1+\beta}{\beta}}   |k-n|^2(1+n)^{\frac{1}{\beta}-1}} <\infty.\nonumber
\end{eqnarray}

Note that $\Lambda=\Gamma$ with $\delta_k=0$.  So Lemmas \ref{lem-estim}   and \ref{integ} imply that $\Lambda$ is a uniqueness set for $\cF^2_\varphi$, that is the unique function in $\cF^2_\varphi$ that vanishes on $\Lambda$ is the zero function. Also, for any $\lambda\in\Lambda$ the sequence $\Lambda\setminus\{\lambda\}$ is a zero sequence of $\cF^2_\varphi$ and, hence,  the system $\cK_\Lambda$ is complete and minimal in $\cF^2_\varphi$ and $\underset{n}{\sum}\ \|h_n-\Bbbk_{\lambda_n}\|^2_{\varphi,2}<\infty$. Bari's theorem  \cite[section A.5.7.1]{N1} ensures that $\cK_\Lambda$ is a Riesz basis for $\cF^2_\varphi$.
\end{proof}
For $\cF^\infty_\varphi$ we have the following result 
\begin{lem}\label{propo}
Let $\lambda^*\in\C\setminus\Lambda$. Then
$\Lambda\cup\{\lambda^*\}$ is a complete interpolating sequence for $\cF^\infty_\varphi$. 
\end{lem}
\begin{proof}
First let us prove that $\widetilde{\Lambda}=\Lambda\cup\{\lambda^*\}$ is a uniqueness set for $\cF^\infty_\varphi$. Indeed, suppose that $\lambda^*\neq0$ and take $f\in\cF^\infty_\varphi$ vanishing on $\Lambda\cup\{\lambda^*\}$, so  that $f(z)=(1-z/\lambda^*)h(z)G_\Lambda(z)$ {where} $h$ is an entire function. By Lemma \ref{lem-estim} we get
$$ |h(z)|\frac{\dist(z,\widetilde{\Lambda})}{(1+|z|)^{1/2+\varepsilon}}\ \lesssim\ |f(z)|\mathrm{e}^{-\varphi(z)}\ \lesssim\ 1,\quad z\in\C.$$
It follows that $h$ is the zero function.
It remains to show that $\widetilde{\Lambda}$ is an  interpolating set for $\cF^\infty_\varphi$. For this purpose let $v=(v_\lambda)_{\lambda\in \widetilde{\Lambda}}$ { be a sequence of complex numbers such that} $\|v\|_{\varphi,\infty,\widetilde{\Lambda}}<\infty.$    Put  $F(z)=(1-z/\lambda^*)G_\Lambda(z)$ and consider the function
\begin{eqnarray*}
F_v(z)\ =\ \sum_{\lambda\in \widetilde{\Lambda}} v_\lambda \frac{F(z)}{F'(\lambda)(z-\lambda)},\quad z\in\C.
\end{eqnarray*}
We have
\begin{eqnarray}
\left|F_v(z)\right| & \leq & \sum_{\lambda\in \widetilde{\Lambda} }\ |v_\lambda| \left|\frac{F(z)}{F'(\lambda)(z-\lambda)}\right| \nonumber\\
   & {\lesssim} & \|v\|_{\varphi,\infty,\widetilde{\Lambda}}\ \sum_{\lambda\in \widetilde{\Lambda}}\ \mathrm{e}^{\varphi(\lambda)}\frac{|z|}{|\lambda|} \left|\frac{G_\Lambda(z)}{G_\Lambda'(\lambda)(z-\lambda)}\right|. \nonumber 
\end{eqnarray}
Let $z\in\C$ and $p\in\N$ such that $|\lambda_{p-1}|\leq |z|<|\lambda_p|$. Suppose that $0\neq |\lambda^*|<|\lambda_1|$ and write $\widetilde{\Lambda}=(\lambda_n)_{n\geq -1}$, where $\lambda_{-1}=\lambda^*$. Let $|z|=e^t$ and $u_n=\log|\lambda_n|$. We have 
$$ \sum_{\lambda\in \widetilde{\Lambda}}\ \mathrm{e}^{\varphi(\lambda)}\frac{|z|}{|\lambda|} \left|\frac{G_\Lambda(z)}{G_\Lambda'(\lambda)(z-\lambda)}\right|  \asymp \underbrace{\sum_{n\leq p-1}\ldots }_{\mathcal{I}_1}+\underbrace{\sum_{n\geq p} \ \mathrm{e}^{\varphi(\lambda_n)} \frac{\dist(z,\Lambda)}{|z-\lambda_n|}\frac{|z|^{p}}{|\lambda_n|^n}\prod_{k=0}^{p-1}\frac{1}{|\lambda_k|}\prod_{k=0}^{n-1}|\lambda_k|}_{\mathcal{I}_2} .$$
By \eqref{mac}, 
\begin{eqnarray*}
\mathcal{I}_1 & = & \sum_{n \leq p-1}\ \mathrm{e}^{\varphi(\lambda_n)} \frac{\dist(z,\Lambda)}{|z-\lambda_n|}\frac{|z|^{p}}{|\lambda_n|^n}\prod_{k=0}^{p-1}\frac{1}{|\lambda_k|}\prod_{k=0}^{n-1}|\lambda_k| \nonumber \\
 & \leq & |z|^{p}\sum_{n \leq p-1}\ \exp\left[u_n^{\beta+1}-(n+1)u_n+ \sum_{k=0}^{n}u_k-\sum_{k=0}^{p-1}u_k\right]\nonumber \\
  & = & \mathrm{e}^{pt-\beta u_{p-1}^{\beta+1}}\ \sum_{n \leq p-1}\ \exp\left[ \left(\frac{u_n}{2}+d_\beta u_n^{1-\beta}\right)-\left(\frac{u_{p-1}}{2}+d_\beta u_{p-1}^{1-\beta}\right)\right]\nonumber \\
  & \leq & \mathrm{e}^{pt-\beta u_{p-1}^{\beta+1}}\ \sum_{n\geq 0}\ \mathrm{e}^{-c \left|u_{p-1}-u_n\right|}\nonumber \\
  & \lesssim & \exp\left(\varphi(z)- \left(\frac{1}{2\beta(\beta+1)}+o(1)\right)(p-s)^2t^{1-\beta}\right) \leq  \mathrm{e}^{\varphi(z)} \nonumber\\
  \end{eqnarray*}
and 
\begin{eqnarray}
\mathcal{I}_2 & \asymp & |z|^{p+1}\sum_{n\geq p}\ \exp\left[u_n^{\beta+1}-(1+n)u_n+ \sum_{k=p}^{n-1}u_k\right] \nonumber \\
  & \leq & \mathrm{e}^{(p+1)t-\beta u_p^{\beta+1}}\ \sum_{n\geq 0}\ \exp\left[-\frac{1}{2}\left| u_n-u_p\right|+d_\beta \left|u_n^{1-\beta}- u_p^{1-\beta}\right|\right] \nonumber \\
  & \lesssim & \mathrm{e}^{(p+1)t-\beta u_p^{\beta+1}}\ \leq \mathrm{e}^{\varphi(z)}. \nonumber  
\end{eqnarray}
Thus, the   interpolating function $F_v$ belongs to  $\cF^\infty_\varphi$. This completes  the proof.

\end{proof}

\subsection{$d-$separated and $\log$-separated sequences}
{A sequence $\left\{\mu_n\right\}$ of real numbers is said to be {\it separated} whenever there exists a constant $\delta>0$ such that  
\[
 \inf_{n\neq m}|\mu_n-\mu_m| \geq\delta.
\]
Let $\Gamma$ be a sequence of complex numbers. It is not difficult to see that if $\log\Gamma:=\left\{\log|\gamma|\ :\ \gamma\in\Gamma\right\}$ is separated then $\Gamma$ is $d-$separated, and also if $\Gamma$ is $d-$separated then $\log\Gamma$ is a finite union of separated sequences. Hence $\Gamma$ is a finite union of $d-$separated sequences if and only if $\log\Gamma$ is a finite union of separated real sequences.}\\

The following lemma was established  by using a Bernstein type inequality in \cite{BDHK} in the case  $\beta=1$.  The proof when $ 0 <\beta <1 $ is different. 
\begin{lem}\label{Boundedness}   Let $\Gamma$  be a sequence defined by  \eqref{Gamm}. Then 
\begin{equation}\label{Gammasep}
\sum_{\gamma_n\in\Gamma}\ \left|\langle f,\Bbbk_{\gamma_n}\rangle\right|^2\ \leq\ C(\Gamma)\ \|f\|_{\varphi,2}^2, 
\end{equation}
for every $f\in\cF^2_\varphi$ if and only if $\left\{\log|\gamma_n| \text{ : } \gamma_n\in \Gamma \right\}$ is a finite union of {separated} sequences.
\end{lem}

\begin{proof}
Let $f$ be a function in $\cF^2_\varphi$. According to Lemma \ref{thm1}, the system $\cK_\Lambda$ is a Riesz basis  for $\cF^2_\varphi$ and, hence, 
\[f(z)\ =\ \sum_{\lambda_n\in \Lambda}\ \langle f,\Bbbk_{\lambda_n}\rangle\ g_{\lambda_n}(z),\quad z\in\C,\]
where $\cG_\Lambda:=\{g_{\lambda_n}\ :\ \lambda_n\in \Lambda\}$ is the unique biorthogonal system of $\cK_\Lambda$ given by 
\begin{equation}\label{glambda}
g_{\lambda_n}(z) = \frac{\|k_{\lambda_n}\|}{G'_\Lambda(\lambda_n)}\ \frac{G_\Lambda(z)}{z-\lambda_n},\qquad z\in\C,
\end{equation}
and  $G_\Lambda$ is the infinite product associated with $\Lambda$. We then get
\begin{eqnarray}
\sum_{m\geq 0}\ \Big|\langle f,\Bbbk_{\gamma_m}\rangle\Big|^2 & = & \sum_{m\geq 0}\ \Big|\Big\langle \sum_{n\geq 0}\ \langle f,\Bbbk_{\lambda_n}\rangle\ g_{\lambda_n},\Bbbk_{\gamma_m}\Big\rangle\Big|^2 \nonumber \\
   & = & \sum_{m\geq 0}\ \Big| \sum_{n\geq 0}\ \langle f,\Bbbk_{\lambda_n}\rangle\ \langle g_{\lambda_n},\Bbbk_{\gamma_m}\rangle\Big|^2.\nonumber 
\end{eqnarray}
Since the operator $f \mapsto \left(\langle f,\Bbbk_{\lambda_n}\rangle\right)$ is an isomorphism between $\cF^2_\varphi$ and $\ell^2$, relation \eqref{Gammasep} is equivalent to the fact that the matrix $\cC= \left(C_{n,m}\right)_{n,m}$ defines a bounded operator on $\ell^2$, where
\begin{eqnarray*}
|C_{n,m}| & = & \left|\langle g_{\lambda_n},\Bbbk_{\gamma_m}\rangle\right| \ =\ \frac{\|\mathrm{k}_{\lambda_n}\|_{\varphi,2}}{\|\mathrm{k}_{\gamma_m}\|_{\varphi,2}}\left|\frac{G_\Lambda(\gamma_m)}{G'_\Lambda(\lambda_n)(\lambda_n-\gamma_m)}\right|.
\end{eqnarray*}
Set $p_m=[(1+\beta)\left(\log|\gamma_m|\right)^\beta]-1$. We have 
$$
\left|G_\Lambda(\gamma_m) \right| \asymp \frac{\dist(\gamma_m,\Lambda)}{\gamma_m}\prod_{k=0}^{p_m}\left|\frac{\gamma_m}{\lambda_k}\right|$$
and 
$$
\left|G'_\Lambda(\lambda_n) \right|  \asymp \frac{1}{|\lambda_n|}\prod_{k=0}^{n}\left|\frac{\lambda_n}{\lambda_k}\right|.
$$
Hence,
\begin{equation}
\label{cnm}
|C_{n,m}|  \asymp \frac{\|\mathrm{k}_{\lambda_n}\|_{\varphi,2}}{\|\mathrm{k}_{\gamma_m}\|_{\varphi,2}}\ \frac{\dist(\gamma_m,\Lambda)}{|\lambda_n-\gamma_m|}\ \frac{|\gamma_m|^{p_m}}{|\lambda_n|^{n}}\prod_{k=0}^{p_m}\frac{1}{|\lambda_k|}\prod_{k=0}^{n}|\lambda_k|.
\end{equation}
By  \eqref{kernelBBB}, we have
\begin{eqnarray*}
\|k_{\lambda_n}\|_{\varphi,2} \asymp \left|e_n(\lambda_n)\right|\quad \text{ and } \quad \|k_{\gamma_m}\|_{\varphi,2} \geq \left|e_{p_m}(\gamma_m)\right|
\end{eqnarray*}
thus we obtain 
\begin{align*}
|C_{n,m}| & \lesssim\  \frac{\|z^{p_m}\|_{\varphi,2}}{\|z^n\|_{\varphi,2}} \frac{\dist(\gamma_m,\Lambda)}{|\lambda_n-\gamma_m|}\ \prod_{k=0}^{p_m}\frac{1}{|\lambda_k|}\ \prod_{k=0}^{n}|\lambda_k| \\
       & \asymp \left(\frac{1+p_m}{1+n}\right)^{\frac{1-\beta}{4\beta}}\ \frac{\dist(\gamma_m,\Lambda)}{|\lambda_n-\gamma_m|}\ e^{-\alpha(n,m)},
\end{align*}
where
\begin{multline*}
\alpha(n,m)  = \beta\Big(\frac{1+n}{1+\beta}\Big)^{\frac{1+\beta}{\beta}}-\beta\Big(\frac{1+p_m}{1+\beta}\Big)^{\frac{1+\beta}{\beta}}+\sum_{k=0}^{p_m} \Big(\frac{1+k}{1+\beta}\Big)^{\frac{1}{\beta}}  -\sum_{k=0}^{n}\Big(\frac{1+k}{1+\beta}\Big)^{\frac{1}{\beta}} \\
     = \frac{1}{2}\Big[ 
   \Big( \frac{1+p_m}{1+\beta}\Big)^{\frac{1}{\beta}} - \Big( \frac{1+n}{1+\beta}\Big)^{\frac{1}{\beta}}\Big] +\frac{1+o(1)}{12\beta(1+\beta)}\Big[ 
   \Big( \frac{1+p_m}{1+\beta}\Big)^{\frac{1-\beta}{\beta}} - \Big( \frac{1+n}{1+\beta}\Big)^{\frac{1-\beta}{\beta}}\Big]. 
\end{multline*}
Now, if $|\gamma_m|\geq| \lambda_n|$,   then we have $\dist(\gamma_m,\Lambda)\leq |\lambda_n-\gamma_m|$, and hence,
\begin{equation}\label{est-cnm}
\left|C_{n,m}\right|\lesssim \mathrm{e}^{-c_\beta\left|(1+p_m)^{\frac{1}{\beta}}-(1+n)^{\frac{1}{\beta}}\right|}.
\end{equation}
If $|\gamma_m|< |\lambda_n|$, then  $\dist(\gamma_m,\Lambda)\leq|\lambda_{p_m}|$ (or $|\gamma_m|)$ and  $|\lambda_n-\gamma_m|\asymp|\lambda_n|$. Again we have  \eqref{est-cnm}. Thus \eqref{est-cnm} holds in both cases.\\

On the other hand, by  Corollary \ref{keernel} and   by \eqref{cnm},  we get
\begin{eqnarray*}
|C_{n,m}| & \gtrsim &\left(\frac{\log|\gamma_m|}{\log|\lambda_n|}\right)^{\frac{1-\beta}{4}}\mathrm{e}^{\varphi(\lambda_n)-\varphi(\gamma_m)}\ \frac{\dist(\gamma_m,\Lambda)}{|\lambda_n-\gamma_m|}\
\frac{ |\gamma_m|^{p_m+1}}{ |\lambda_n|^{n+1}} \prod_{k=0}^{p_m}\frac{1}{|\lambda_k|}\prod_{k=0}^{n}|\lambda_k| \\
 & \asymp &\left(\frac{1+p_m}{1+n}\right)^{\frac{1-\beta}{4\beta}}\frac{\dist(\gamma_m,\Lambda)}{|\lambda_n-\gamma_m|}\ \mathrm{e}^{-\widetilde{\alpha}(n,m)},
\end{eqnarray*}
where
{\begin{multline*}
\widetilde{\alpha}(n,m)  =\Big(\frac{1+p_m}{1+\beta}\Big)^{\frac{1+\beta}{\beta}}
-\Big(\frac{1+n}{1+\beta}\Big)^{\frac{1+\beta}{\beta}}\\
 +(1+n)\Big(\frac{1+n}{1+\beta}\Big)^{\frac{1}{\beta}}
 - (1+p_m)  \Big(\frac{1+p_m}{1+\beta}\Big)^{\frac{1+\beta}{\beta}}
+\sum_{k=0}^{p_m}  \Big(\frac{1+k}{1+\beta}\Big)^{\frac{1}{\beta}}
     -\sum_{k=0}^{n}   \Big(\frac{1+k}{1+\beta}\Big)^{\frac{1}{\beta}} \\
      =  \frac{1}{2}\Big[ 
   \Big( \frac{1+p_m}{1+\beta}\Big)^{\frac{1}{\beta}} - \Big( \frac{1+n}{1+\beta}\Big)^{\frac{1}{\beta}}\Big] +\frac{1+o(1)}{12\beta(1+\beta)}\Big[ 
   \Big( \frac{1+p_m}{1+\beta}\Big)^{\frac{1-\beta}{\beta}} - \Big( \frac{1+n}{1+\beta}\Big)^{\frac{1-\beta}{\beta}}\Big]. 
\end{multline*}
}
In the same way we get
\[\left|C_{n,m}\right| \gtrsim \mathrm{e}^{-\widetilde{c}_\beta\left|(1+p_m)^{\frac{1}{\beta}}-(1+n)^{\frac{1}{\beta}}\right|}.\]
Finally,
\[\mathrm{e}^{-\widetilde{c}_\beta\left|(1+p_m)^{\frac{1}{\beta}}-(1+n)^{\frac{1}{\beta}}\right|}\ \lesssim\  \left|C_{n,m}\right|\ \lesssim\  \mathrm{e}^{-c_\beta\left|(1+p_m)^{\frac{1}{\beta}}-(1+n)^{\frac{1}{\beta}}\right|}.\]

If $\big\{\log|\gamma| \text{ : } \gamma \in \Gamma \big\}$ is a finite union of separated sequences, then  $C=\left(C_{n,m}\right)$ is bounded on $\ell^2$. \\

In the opposite direction, since $\mathcal{K}_\Lambda$ is a Riesz basis, let $g_\lambda$ be the function given in \eqref{glambda},   $\|g_\lambda\|^{2}_{\varphi,2}\asymp 1$ and as before we have 
\begin{eqnarray*}
1 & \gtrsim & \sum_{\gamma\in\Gamma}\frac{|g_\lambda(\gamma)|^2}{k_\gamma(\gamma)}\\
  &\gtrsim  & \sum_{\gamma\in\Gamma \cap D_d(\lambda,\delta)} \frac{|g_\lambda(\gamma)|^2}{k_\gamma(\gamma)} \\
  & \gtrsim & \sum_{\gamma\in\Gamma \cap D_d(\lambda,\delta)} \mathrm{e}^{-c_\beta |\log|\gamma|-\log|\lambda| |}\\
  &\asymp   & \mathrm{Card} (\Gamma \cap  D_d(\lambda,\delta)). 
\end{eqnarray*}
Since $\dist(\gamma,\Lambda)\asymp |\gamma-\lambda|$, we get
$$\sup_{\lambda\in \Lambda} \mathrm{Card} (\Gamma \cap D_d(\lambda,\delta))<\infty.$$ 
Therefore $\Gamma$ is a finite union of $d-$separated sequences.

\end{proof}


\section{Proof of Theorem \ref{thm30}}\label{sectthree}

\begin{proof}
"$\Longleftarrow$"

 First of all, since the sequence {$\Gamma$ is $d-$separated},  by Lemma \ref{Boundedness} the operator $T_\Gamma$ is bounded from $\cF^2_\varphi$ to $\ell^2$. On the other hand, the sequence $\Gamma$ is a uniqueness set for  $\cF^2_\varphi$, and for any $\gamma\in \Gamma$, the sequence $\Gamma\setminus\{\gamma\}$ is a zero set for $\cF^2_\varphi$. Indeed, let $f\in\cF^2_\varphi$ be such that $f|_{\Gamma}=0$. By Hadamard's factorization theorem \cite{Lev} we can write $f=hG_\Gamma$, for an entire function $h$. Since  $f\in\cF^2_\varphi$, by Lemma \ref{keernel} we have 
 $$|f(z)|\leqslant \|f\|_{\varphi,2}\|\mathrm{k}_z\|_{\varphi,2}\lesssim \frac{\mathrm{e}^{\varphi(z)}}{|z|(\log|z|)^{\frac{1-\beta}{4}}}{,}\quad |z|>1.$$ 
  Lemma \ref{lem-estim} implies that
\begin{eqnarray*}
  |h(z)|\frac{\dist(z,\Gamma)}{(1+|z|)^{3/2+c_N}}\ \lesssim\ |f(z)|\mathrm{e}^{-\varphi(z)}\ \lesssim\ \frac{1}{|z|(\log|z|)^{\frac{1-\beta}{4}}} ,\quad |z|>1,
\end{eqnarray*}
where $c_N=\Delta_N +\varepsilon$, for a small enough $\varepsilon$ {so that} $c_N<1/2$. Then $h$ must be a polynomial. Combining this fact with 
 Lemma \ref{integ}, we conclude that $h$ must be identically zero, and hence $\Gamma$ is a uniqueness set for $\cF^2_\varphi$, therefore $T_\Gamma$ is injective. Futhermore {$T_\Gamma g_{n} =e_{n+1}$, $n\geq 0$, where $e_{n+1}$ is  the $(n+1)-$th element of the canonical basis of $\ell^2$,} and  
 $$\displaystyle g_n(z)=\frac{\|\mathrm{k}_{\gamma_n}\|_{\varphi,2}}{G'_{\Gamma}(\gamma_n)}\frac{G_\Gamma(z)}{z-\gamma_n}, \qquad z\in \C.$$
Therefore,  the range of $T_\Gamma$ is dense in $\ell^2$. In order to prove that  $T_\Gamma$ is onto, we associate with each $a\in\ell^2$ the function $H_a$ as follows :

\begin{eqnarray*}\label{ss}
H_a(z)=\sum_{\gamma\in\Gamma}\ a_\gamma\ \frac{\|\mathrm{k}_\gamma\|_{\varphi,2}}{G_\Gamma'(\gamma)}\frac{G_\Gamma(z)}{z-\gamma},\quad z\in\C.
\end{eqnarray*}
The estimate of the reproducing kernel given in Corollary \ref{keernel} and the estimate on $G_\Gamma$ proved in Lemma \ref{lem-estim} ensure that the above series converges uniformly on every compact set of $\C$. It remains now to show that $H_a\in\cF^2_\varphi$. Indeed, since $\cK_\Lambda$ is a Riesz basis  for $\cF^2_\varphi$ (see Lemma \ref{thm1}), we have
\begin{eqnarray}
\|H_a\|_{\varphi,2}^2 & \asymp & \sum_{m\geqslant 0}\left|\langle H_a,\Bbbk_{\lambda_m}\rangle\right|^2 \nonumber \\
   & = & \sum_{m\geqslant 0}\left|\sum_{n\geqslant 0}\ a_n\ \frac{G_\Gamma(\lambda_m)}{G'_\Gamma(\gamma_n)(\lambda_m-\gamma_n)}\frac{\|\mathrm{k}_{\gamma_n}\|_{\varphi,2}}{\|\mathrm{k}_{\lambda_m}\|_{\varphi,2}}\right|^2. \nonumber
\end{eqnarray}
Hence $\{H_a\}_{a\in\ell^2}\subset\cF^2_\varphi$ if and only if the matrix $A=\left(A_{n,m}\right)_{n,m\geqslant 0}$ defines a bounded operator on $\ell^2$, where
\[
\left|A_{n,m}\right|  =  \left| \frac{G_\Gamma(\lambda_m)}{G'_\Gamma(\gamma_n)(\lambda_m-\gamma_n)}\frac{\|\mathrm{k}_{\gamma_n}\|_{\varphi,2}}{\|\mathrm{k}_{\lambda_m}\|_{\varphi,2}} \right| .
\]
Recall that for every $z\in\C$ there exists $p\in\N$ such that $|\gamma_{p-1}|\leq |z|<|\gamma_p|$. We  have
$$
\left| G_\Gamma(z)\right|  \asymp\ \frac{\dist(z,\Gamma)}{|z|}\ \prod_{k=0}^{p-1}\ \left|z/\gamma_k\right| \quad 
$$
and 
$$
\left| G'_\Gamma(\gamma_n)\right|  \asymp\ \frac{1}{|\gamma_n|}\ \prod_{k=0}^{n-1}\ \left|\frac{\gamma_n}{\gamma_k}\right|. $$
Since $\delta_n=O\big((1+n)^{\frac{1}{\beta}-1}\big),$ there exists $M>0$ such that for every $m\geq 0$ there exists $|i|\leq M$ such that  the index $p$  corresponding to $\lambda_m$ is $m+i$. Consequently,
\[
\left|A_{n,m}\right|\ \asymp\ \frac{\|\mathrm{k}_{\gamma_n}\|_{\varphi,2}}{\|\mathrm{k}_{\lambda_m}\|_{\varphi,2}}\  \frac{\dist(\lambda_m,\Gamma)}{|\lambda_m-\gamma_n|}\ \frac{|\gamma_n|}{\lambda_m}\ \prod_{k=0}^{m+i-1}\ \left|\frac{\lambda_m}{\gamma_k}\right|\prod_{k=0}^{n-1}\ \left|\frac{\gamma_k}{\gamma_n}\right|. \]
By Lemma \ref{Kernel_Estimate} we have
\[\|k_z\|^2_{\varphi,2}\ \asymp\ \left|e_{n_z-1}(z)\right|^2 + \left|e_{n_z}(z)\right|^2 + \frac{\mathrm{e}^{2\varphi(z)}}{\rho(z)^2}, \]
where $n_z=[(1+\beta)\log^\beta|z|]$. Again, since $\delta_n=O((1+n)^{\frac{1}{\beta}-1})$, for every $n\geq 0$ there exists $|j|\leq M$ satisfying
\[\|k_{\gamma_n}\|^2_{\varphi,2}\ \asymp\ \left|e_{n+j}(\gamma_n)\right|^2 + \frac{\mathrm{e}^{2\varphi(\gamma_n)}}{|\gamma_n|^2\left(\log| \gamma_n|\right)^{1-\beta}}. \]
Also for every $m\geq 0$ we have
\[\|k_{\lambda_m}\|_{\varphi,2}\ \asymp\ \left|e_{m}(\lambda_m)\right|\ \asymp\ \frac{\mathrm{e}^{\varphi(\lambda_n)}}{|\lambda_m|\left(\log |\lambda_m|\right)^{\frac{1-\beta}{4}}}. \]
By these estimates we can write
$|A_{n,m}|\ \asymp\ I_{n,m} + J_{n,m},$ where 
\begin{eqnarray*}
I_{n,m} & \asymp& \frac{|e_{n+j}(\gamma_n)|}{|e_m(\lambda_m)|}\  \frac{|\lambda_m|}{|\lambda_m-\gamma_n|}\ \frac{|\gamma_n|}{|\lambda_m|}\ \prod_{k=0}^{m+i-1}\ \left|\frac{\lambda_m}{\gamma_k}\right|\prod_{k=0}^{n-1}\ \left|\frac{\gamma_k}{\gamma_n}\right|\nonumber\\ 
&=:& \ \frac{|\lambda_m|}{|\lambda_m-\gamma_n|} e^{\Theta(n,m)} \nonumber
\end{eqnarray*}
and
\begin{eqnarray*}
J_{n,m} & \asymp&  \frac{\left(\log|\lambda_m|\right)^{\frac{1-\beta}{4}}}{\left(\log|\gamma_n|\right)^{\frac{1-\beta}{2}}}\ \mathrm{e}^{\varphi(\gamma_n)-\varphi(\lambda_m)}\  \frac{|\lambda_m|}{|\lambda_m-\gamma_n|}\ \prod_{k=0}^{m+i-1}\ \left|\frac{\lambda_m}{\gamma_k}\right|\prod_{k=0}^{n-1}\ \left|\frac{\gamma_k}{\gamma_n}\right|\nonumber\\
&  =:& \frac{|\lambda_m|}{|\lambda_m-\gamma_n|}e^{\theta(n,m)}. \nonumber
\end{eqnarray*}
To estimate the coefficients $I_{n,m}$ and $J_{n,m}$, we first write
$$K_{n,m}\ =\ \underset{k=0}{\overset{m+i-1}{\prod}}\ \left|\frac{\lambda_m}{\gamma_k}\right|\underset{k=0}{\overset{n-1}{\prod}}\ \left|\frac{\gamma_k}{\gamma_n}\right|.$$
We next put $u_k:=\log|\lambda_k|$ for every $k\geq 0$. We have
\begin{multline*}
\log K_{n,m}  =  (m+i)u_m-n(u_n+\delta_n) -\sum_{k=0}^{m+i-1} u_k + \sum_{k=0}^{n-1} u_k -\sum_{k=0}^{m+i-1} \delta_k + \sum_{k=0}^{n-1} \delta_k  \nonumber\\
  =  \frac{u_{m+i}-u_n}{2}+\left[\beta u_n^{\beta+1}-n(u_n+\delta_n)\right] - \left[\beta u_{m+i}^{\beta+1}-(m+i)u_m\right] \nonumber\\
 +d_\beta \left(u_n^{1-\beta}-u_{m+i}^{1-\beta}\right)-\sum_{k=0}^{m+i-1} \delta_k + \sum_{k=0}^{n-1} \delta_k.  \nonumber
\end{multline*}
Hence, 
\begin{multline*}
\Theta(n,m) :=  \log K_{n,m} + \log|e_{n+j}(\gamma_n)|-\log e_m(\lambda_m) + (u_n+\delta_n)-u_m \nonumber\\
  =   \frac{u_n-u_{m+i}}{2} +\left(\frac{1}{2\beta(\beta+1)}+o(1)\right)\left(j^2u_n^{1-\beta}-i^2u_{m+i}^{1-\beta}\right)+ d_\beta \left(u_n^{1-\beta}-u_{m+i}^{1-\beta}\right) \nonumber \\
     -\sum_{k=0}^{m+i-1} \delta_k + \sum_{k=0}^{n} \delta_k  + \frac{1-\beta}{4}\log\frac{1+m}{1+n}.
\end{multline*} 
Now if $n\leq m+i$, then
$$ I_{n,m}\ \asymp\  \frac{|\lambda_m|}{|\lambda_m-\gamma_n|} \mathrm{e}^{\Theta(n,m)}\ \asymp\  \mathrm{e}^{\Theta(n,m)} $$
and $$\Theta(n,m) = -\left(\frac{1}{2}+o(1)\right)(u_{m+i}-u_n) - \sum_{k=n+1}^{m+i-1} \delta_k.$$
If $n\geq m+i$, then 
$$I_{n,m}\ \asymp\  \frac{|\lambda_m|}{|\lambda_m-\gamma_n|} \mathrm{e}^{\Theta(n,m)}\ \asymp\ \mathrm{e}^{\Theta(n,m)+u_m-u_n-\delta_n}$$
and \[\Theta(n,m)+u_m-u_n-\delta_n= -\left(\frac{1}{2}+o(1)\right)\left(u_n-u_{m+i}\right) + \sum_{k=m+i}^{n-1} \delta_k.\]
Hence,
\[I_{n,m}\ \asymp\ \exp\left[-\left(\frac{1}{2}+o(1)\right)\left|u_n-u_{m+i}\right| \pm \sum_{k=m+i}^{n-1} \delta_k \right].\]
Similarly, we have 
\begin{multline*}
\theta(n,m) :=  \log K_{n,m}+\varphi(\gamma_n)-\varphi(\lambda_m) + \frac{1-\beta}{4}\log\frac{1+m}{(1+n)^2} \\
         =  \frac{u_n-u_{m+i}}{2} + \left(\frac{1}{2\beta(\beta+1)}+o(1)\right)\left(j^2u_n^{1-\beta}-i^2u_{m+i}^{1-\beta}\right)+  \frac{1-\beta}{4}\log\frac{1+m}{(1+n)^2}\\
       + d_\beta \left(u_n^{1-\beta}-u_{m+i}^{1-\beta}\right)-\sum_{k=0}^{m+i-1} \delta_k + \sum_{k=0}^{n} \delta_k.     
\end{multline*} 
If $n\leq m+i$, we have 
$$J_{n,m}\ \asymp \frac{|\lambda_m|}{|\lambda_m-\gamma_n|}\mathrm{e}^{\theta(n,m)}\ \asymp \mathrm{e}^{\theta(n,m)},$$
 and
\begin{equation}
 \theta(n,m) = -\left(\frac{1}{2}+o(1)\right)\left(u_{m+i}-u_n\right)-\sum_{k=n+1}^{m+i-1} \delta_k. \nonumber
\end{equation}
If $n\geq m+i$, then 
$$J_{n,m}\ \asymp \frac{|\lambda_m|}{|\lambda_m-\gamma_n|}\mathrm{e}^{\theta(n,m)}\ \asymp \mathrm{e}^{\theta(n,m)+u_m-u_n-\delta_n}$$ and 
\[\theta(n,m)+u_m-u_n-\delta_n= -\left(\frac{1}{2}+o(1)\right)\left(u_n-u_{m+i}\right)+\sum_{k=m+i}^{n-1} \delta_k.\]
Hence,
\[J_{n,m}\ \asymp\ \exp\left[-\left(\frac{1}{2}+o(1)\right)\left|u_n-u_{m+i}\right|\pm\sum_{k=m+i}^{n-1} \delta_k\right].\]
Consequently,
\begin{eqnarray}
 |A_{n,m}|\ & \asymp\ & I_{n,m} + J_{n,m}\nonumber \ \asymp\ \exp\left[-\left(\frac{1}{2}+o(1)\right)\left|u_n-u_{m+i}\right|\pm\sum_{k=m+i}^{n-1} \delta_k\right]. \nonumber
\end{eqnarray}
Recall that
\[\Delta_N := \limsup_{n}\ \frac{1}{u_{n+N}-u_{n}}\ \left|\sum_{k=n+1}^{n+M}\ \delta_k\right|\ <\ \frac{1}{2}. \]
This implies that for a very small $\varepsilon$ (chosen in such a way that $\Delta_N+\varepsilon<1/2$) and for every $n\geq 0$ we have
\[ \left|\sum_{k=n+1}^{n+M}\ \delta_k\right|\ \leq\  \left(\Delta_N+\frac{\varepsilon}{2}\right)\left(u_{n+N}-u_{n}\right). \]
Thus,
\begin{eqnarray*}
 |A_{n,m}|\ 
 & \lesssim & \exp\left[-\left(\frac{1}{2}-\left(\Delta_N+\frac{\varepsilon}{2}\right)+o(1)\right)\left|u_n-u_{m+i}\right| \right].
\end{eqnarray*}
Since $\Delta_N<\frac{1}{2}$, the matrix $A= \left(A_{n,m}\right)$ defines a bounded operator on $\ell^2$ and, hence, $\cK_\Gamma$ is a Riesz basis for $\cF^2_\varphi$.\\

"$\Longrightarrow$"

Proof of \eqref{first}. Let  $\varphi(r)=\varphi_\beta(r)=\left(\log^+r\right)^{1+\beta}$, $0<\beta\leq 1$. If $\mathcal{K}_{\Gamma}$  is a Riesz  {basis} for $\cF_{\varphi_\beta}^{2}$ then $\Gamma$ is an interpolating sequence for $\cF_{\varphi_\beta}^{2}$ and hence for  $\cF_{\varphi_1}^{2}$.  By \cite[Corollary 2.3]{BDHK}  {$\Gamma$ is $d-$separated.} \\

Proof of \eqref{second}.  
Let $\Gamma=\{\gamma_n\}$ be a sequence of complex numbers such that $\cK_\Gamma$ is a Riesz basis for  $\cF^2_\varphi$. Then for every $\gamma\in\Gamma$ there exists a unique function $f_\gamma\in\cF^2_\varphi$ that satisfies the interpolation problem : 
$$\langle f_\gamma, \Bbbk_{\gamma} \rangle = 1 \quad \text{ and }  \quad \langle f_\gamma, \Bbbk_{\gamma} \rangle =0,\quad \gamma'\in \Gamma\setminus\{\gamma\}.$$
 Consequently, $\Gamma\setminus\{\gamma\}$ is a subset of the zero set of the function $f_\gamma$. Since $\mathcal{K}_\Gamma$ is complete, $\Gamma$ is a uniqueness set, then $\Gamma\setminus\{\gamma\}$ is exactly the zero set of $f_\gamma$. By  Hadamard's factorization theorem we have 
$f_\gamma(z) = c\frac{G_\Gamma(z)}{G'_\Gamma(\gamma)(z-\gamma)},$ for some constant $c\in\C$. Since $\langle f_\gamma, \Bbbk_{\gamma} \rangle = 1$, we  get $c=\|k_\gamma\|_{\varphi,2}$. Therefore, 
$$f_\gamma(z) = \|k_\gamma\|_{\varphi,2}\ \frac{G_\Gamma(z)}{G'_\Gamma(\gamma)(z-\gamma)}, \quad z\in\C.$$
Furthermore $\|f_\gamma\|_{\varphi,2}\asymp 1$. Hence,
\begin{eqnarray*}
\left| f_\gamma(\lambda_m) \right|\ \leq\ \|f_\gamma\|_{\varphi,2}\ \|k_{\lambda_m}\|\ \asymp\ \left|e_m(\lambda_m)\right|,
\end{eqnarray*}
for every $\lambda_m\in\Lambda$.\\

Suppose now that the sequence $(\delta_n/(1+n)^{\frac{1}{\beta}-1})$ is unbounded. Without loss of generality we assume the existence of a subsequence $\big(\delta_{n_k}/(1+n_k)^{\frac{1}{\beta}-1}\big)$ which tends to $+\infty$ (the case of  convergence to $-\infty$ is similar). Then, for every $k$ there exists $m_k$ such that $|n_k-m_k|\rightarrow\infty$ and $\gamma_{n_k}$ is close to $\lambda_{m_k}$.  
We obtain
\begin{eqnarray}\label{estima}
\left| f_{\gamma_{n_k}}(\lambda_{m_k}) \right|\ \lesssim\ \|k_{\lambda_{m_k}}\|_{\varphi,2}\ \asymp\ \left|e_{m_k}(\lambda_{m_k})\right|.
\end{eqnarray}
On the other hand, we have $
\|k_{\gamma_{n_k}}\|_{\varphi,2} \geq \ \left|e_{m_k}(\gamma_{n_k})\right|$ and

\begin{eqnarray}
\left|G_\Gamma(\lambda_{m_k})\right| \asymp\ \frac{|\lambda_{m_k}-\gamma_{n_k}|}{|\lambda_{m_k}|}\ \prod_{j=0}^{n_k-1} \left|\frac{\lambda_{m_k}}{\gamma_j}\right|,\quad \left|G'_\Gamma(\gamma_{n_k})\right| \asymp\ \frac{1}{|\gamma_{n_k}|}\ \prod_{j=0}^{n_k-1} \left|\frac{\gamma_{n_k}}{\gamma_j}\right|. \nonumber
\end{eqnarray}
For the sake of brevity, we denote $n:=n_k$ and $m:=m_k$. Identity \eqref{estima} becomes
\begin{eqnarray*}
\left|e_m(\lambda_m)\right| & \gtrsim & \left|f_{\gamma_n}(\lambda_m)\right|\ \gtrsim\ \frac{|\gamma_n|}{|\lambda_m|}\left|e_m(\gamma_n)\right|\prod_{j=0}^{n-1} \left|\frac{\lambda_{m}}{\gamma_j}\right|\left|\frac{\gamma_j}{\gamma_{n}}\right|. 
\end{eqnarray*}
Therefore,
\begin{eqnarray}
\left|\frac{e_m(\lambda_m)}{e_m(\gamma_n)}\right|  &  \gtrsim  &  \left|\frac{\lambda_m}{\gamma_n}\right|^{n-1}\ \iff\ \left|\frac{\lambda_m}{\gamma_n}\right|^{n-m-1} \lesssim\ 1. \label{boun}
\end{eqnarray}
Note that we can suppose  $|\lambda_m|\leq |\gamma_n|$  and, hence, $m=n+\delta'_n+\delta_{n,m}$, where $\left(\delta'_n\right)$ is a sequence  tending to infinity $\left(\delta'_n=\mbox{Const}(\beta)\delta_n/(1+n)^{\frac{1}{\beta}-1}\right)$ and $\left(\delta_{n,m}\right)$ is a bounded negative sequence such that $\delta_{n,m}\leq- 1$ (otherwise we replace $m$ by $m-m'$ for a suitable integer $m'$). Thus,
\begin{eqnarray*}
\log \left|\frac{\lambda_m}{\gamma_n}\right| & = & u_m - \left(u_n+\delta_n\right) \nonumber\\
     & = & b_\beta \left((1+\beta)(u_n+\delta_n)^\beta+\delta_{n,m}\right)^{\frac{1}{\beta}} - \left(u_n+\delta_n\right) \nonumber\\
     & = &  \left(\frac{1}{\beta(\beta+1)}+o(1) \right)\delta_{n,m}(u_n+\delta_n)^{1-\beta}.
\end{eqnarray*}
Since $n-m-1=n+\delta'_n-m-\delta'_n-1=-\delta_{n,m}-\delta_n-1$, 
\eqref{boun} becomes
\begin{eqnarray}
1 & \gtrsim & -(\delta_n+\delta_{n,m}+1) \left(\frac{1}{\beta(\beta+1)}+o(1) \right)\delta_{n,m}(u_n+\delta_n)^{1-\beta} \nonumber \ \geq \  c_\beta\ \delta_n u_n^{1-\beta}, 
\end{eqnarray}
which is impossible.\\

Proof of \eqref{third}. Recall first that the matrix of the coefficients
\begin{eqnarray*}
\left|A_{n,m}\right|  &=&  \left|\frac{G_\Gamma(\lambda_m)}{G'_\Gamma(\gamma_n)(\lambda_m-\gamma_n)}\frac{\|k_{\gamma_n}\|_{\varphi,2}}{\|k_{\lambda_m}\|_{\varphi,2}}\right|
\end{eqnarray*}
defines a bounded operator on $\ell^2$. From the proof of the first part we have
\begin{eqnarray*}
\left|A_{n,m}\right| & \gtrsim & \exp\left[-\frac{\left|u_{m+i}-u_n\right|}{2}+c_\beta\left(u_n^{1-\beta}-u_{m+i}^{1-\beta}\right) \mp \sum_{k=n+1}^{m+i-1}\delta_k\right].
\end{eqnarray*}
 Assume that for every $N\geq 1$, we have
$$\Delta_N :=\limsup_n \frac{1}{u_{n+N}-u_{n}}\left|\sum_{k=n+1}^{n+N}\ \delta_k\right| = \frac{1}{2}+\varepsilon_N,$$
for a nonnegative sequence $\left(\varepsilon_N\right)$. For every $N\geq 1$ there exists an integer $n_N$ (sufficiently large) such that 
\begin{eqnarray*}
\left|\sum_{k=n_N+1}^{n_N+N}\ \delta_k\right| \ \geq\  \frac{u_{n_N+N}-u_{n_N}}{2} + \varepsilon_N\frac{u_{n_N+N}-u_{n_N}}{2}.
\end{eqnarray*}
$\bullet$ Assume that there exists a subsequence $(N_l)$ such that $\underset{k=n_l+1}{\overset{n_l+N_l}{\sum}}\ \delta_k \ >0$ ($n_l$ is the integer $n_{N_l}$).  We  get 
\begin{eqnarray*}
\left|A_{n_l+N_l,n_l}\right| & \gtrsim & \exp\left[-\frac{\left|u_{n_l}-u_{n_l+N_l}\right|}{2}+c_\beta\left(u_{n_l+N_l}^{1-\beta}-u_{n_l}^{1-\beta}\right)+\sum_{k=n_l+1}^{n_l+N_l}\delta_k\right]\nonumber\\
& \gtrsim &  \exp\left[c_\beta\left(u_{n_l+N_l}^{1-\beta}-u_{n_l}^{1-\beta}\right)\right].
\end{eqnarray*}
This implies that $\left(A_{n_l+N_l,n_l}\right)$ is unbounded $(0< \beta< 1)$ and, hence, the matrix $\left(A_{n,m}\right)$ cannot represent a bounded operator on $\ell^2$.\\
$\bullet$ There exists $N_0\geq 1$ such that for every $N\geq N_0$ we have $\underset{k=n_l+1}{\overset{n_l+N_l}{\sum}}\ \delta_k <0$. For every $n$ we have \begin{eqnarray*}
\left|A_{n,n+N}\right| 
& \gtrsim & \exp\left[-\frac{u_{n+N}-u_n}{2}-c_\beta\left(u_{n+N}^{1-\beta}-u_n^{1-\beta}\right) - \sum_{k=n+1}^{n+N}\delta_k\right].
\end{eqnarray*}
If $(\varepsilon_N)$ contains a subsequence $\left(\varepsilon_N\right)_{N\in J}$ which is bounded below by some $\varepsilon>0$,  where $J$ is an infinite subset of $\N$. Then for $N\in J$, we have
\begin{eqnarray}
\left|A_{n_N,n_N+N}\right| & \gtrsim & \exp\left[-c_\beta\left(u_{n_N+N}^{1-\beta}-u_{n_N}^{1-\beta}\right) + \varepsilon\frac{u_{n_N+N}-u_{n_N}}{2}\right]. \nonumber 
\end{eqnarray}
This ensures that $\left(A_{n_N,n_N+N}\right)$ tends to infinity. Therefore,  the matrix $A$ cannot define a bounded operator on $\ell^2$. Suppose now that $(\varepsilon_N)$ converges to zero. Simple computations yield
\begin{eqnarray}
\left|A_{n_N,n_N+N}\right| & \gtrsim & \exp\left[-c_\beta\left(u_{n_N+N}^{1-\beta}-u_{n_N}^{1-\beta}\right) + \varepsilon_N\frac{u_{n_N+N}-u_{n_N}}{2}\right] \nonumber \\
 & \gtrsim & \exp\left[-c_\beta\left(u_{n_N+N}^{1-\beta}-u_{n_N}^{1-\beta}\right) + \frac{\varepsilon_N}{2}\frac{\beta(1+n_N)}{1-\beta^2}\left(u_{n_N+N}^{1-\beta}-u_{n_N}^{1-\beta}\right)\right] \nonumber \\
 & = & \exp\left[\left(u_{n_N+N}^{1-\beta}-u_{n_N}^{1-\beta}\right)\left(\frac{\varepsilon_N}{2}\frac{\beta(1+n_N)}{1-\beta^2}-c_\beta\right)\right]. \nonumber
\end{eqnarray}
Choose  $n_N$  such that  $\varepsilon_N(1+n_N)>\frac{4c_\beta}{\beta}(1-\beta^2),\ (0<\beta< 1).$
It follows  that
\begin{eqnarray*}
\left|A_{n_N,n_N+N}\right| & \gtrsim &  \exp\left[\left(u_{n_N+N}^{1-\beta}-u_{n_N}^{1-\beta}\right)\left(2c_\beta-c_\beta\right)\right]  \nonumber\\
 & = & \exp\left[c_\beta\left(u_{n_N+N}^{1-\beta}-u_{n_N}^{1-\beta}\right)\right].
\end{eqnarray*}
Thus $\left(A_{n_N,n_N+N}\right)_N$ is unbounded and, hence, the matrix $A=\left(A_{n,m}\right)$ cannot define a bounded operator on $\ell^2$. This  completes the proof.

\end{proof}

\section{Proof of Theorem \ref{thmInfty}}\label{sectioninftinity}
$"\Longleftarrow"$

Suppose that $\cK_\Gamma$ is a Riesz basis for $\cF^2_\varphi.$ {Then} $\Gamma$ satisfies conditions \eqref{first}, \eqref{second} and \eqref{third} of Theorem \ref{thm30}.

First, the sequence $\widetilde{\Gamma}=\Gamma\cup\{\gamma^*\}$ is a uniqueness set for $\cF^\infty_\varphi$. Indeed, if $f$ is a function from $\cF^\infty_\varphi$ that vanishes on $\widetilde{\Gamma}$, then $f=(1-z/\gamma^*) G_\Gamma h$, for an entire function $h$. Our estimates of $G_\Gamma$ imply that 
\begin{eqnarray}
|h(z)|\frac{\dist(z,\Gamma)}{(1+|z|)^{1/2+\Delta_N+\varepsilon}} \lesssim |f(z)|\mathrm{e}^{-\varphi(z)} \lesssim 1,\quad z\in\C. \nonumber
\end{eqnarray}
Hence,
$$|h(z)| \lesssim\ \frac{(1+|z|)^{1/2+\Delta_N+\varepsilon}}{\dist(z,\Gamma)},\quad z\in\C\setminus\Gamma.$$
Since $\Delta_N <1/2$, the function $h$ must be identically zero. \\

Let us prove that $\widetilde{\Gamma}=\{\gamma_n\}_{n\geq -1}$ is an interpolating sequence for $\cF^\infty_\varphi$, where $\gamma_{-1}=\gamma^*$. For this, let $v=(v_n)$ be a sequence such that 
$\|v\|_{\varphi,\infty,\widetilde{\Gamma}}<\infty$ and consider the  entire function  $L_v$, 
\begin{eqnarray*}
L_v(z)\ =\ \sum_{n\geq -1} v_n \frac{F_\Gamma(z)}{F'_\Gamma(\gamma_n)(z-\gamma_n)},\quad z\in\C.
\end{eqnarray*}
where $F_\Gamma=(1-z/\gamma^*) G_\Gamma$. Let us verify that $L_v\in\cF^\infty_\varphi$. According to Lemma \ref{propo}, $\Lambda\cup\{\lambda^*\}$ is a complete interpolating set for $\cF^\infty_\varphi$ and consequently
$$
\|L_v\|_{\varphi,\infty}  \asymp   \sup_{m\geq -1} \mathrm{e}^{-\varphi(\lambda_m)}\left|\sum_{n\geq 0} v_n \frac{F_\Gamma(\lambda_m)}{F'_\Gamma(\gamma_n)(\lambda_m-\gamma_n)}\right| 
=\sup_{m\geq -1} \left|\sum_{n\geq 0} v_n \mathrm{e}^{-\varphi(\gamma_n)} B_{n,m} \right|,  
$$ 
where
\begin{eqnarray*}
\left|B_{n,m}\right| & = &  \mathrm{e}^{\varphi(\gamma_n)-\varphi(\lambda_m)} \left|\frac{F_\Gamma(\lambda_m)}{F'_\Gamma(\gamma_n)(\lambda_m-\gamma_n)}\right| \nonumber\\
 & \asymp & \mathrm{e}^{\varphi(\gamma_n)-\varphi(\lambda_m)} \frac{|\lambda_m|}{|\gamma_n|} \left|\frac{G_\Gamma(\lambda_m)}{G'_\Gamma(\gamma_n)(\lambda_m-\gamma_n)}\right| \nonumber \\
 & \asymp & \left|A_{n,m}\right|\mathrm{e}^{o(1)|u_m-u_n|}, 
\end{eqnarray*}
and  $u_n=\log|\lambda_n|$. The estimates on the matrix $A=\left(A_{n,m}\right)$ imply that $L_v$ belongs to $\cF^\infty_\varphi$.\\

$"\Longrightarrow"$ 

Suppose now that $\Gamma\cup\{\gamma^*\}$ is a complete interpolating set for $\cF^\infty_\varphi$. To prove that $\cK_\Gamma$ is a Riesz basis for $\cF^2_\varphi$, it suffices to verify that $\Gamma$ satisfies conditions \eqref{first}-\eqref{third} of Theorem \ref{thm30}. First  $\Gamma$ is $d-$separated because every interpolating sequence for $\cF^\infty_{\varphi_\beta}$ is also an interpolating sequence for $\cF^\infty_{\varphi_1}$. Furthermore, remark that 
\begin{eqnarray}
\left|B_{n,m}\right| \asymp \mathrm{e}^{\varphi(\gamma_n)-\varphi(\lambda_m)} \frac{|\lambda_m|}{|\gamma_n|} \left|\frac{G_\Gamma(\lambda_m)}{G'_\Gamma(\gamma_n)(\lambda_m-\gamma_n)}\right| \gtrsim \frac{(\log \gamma_n)^{\frac{1-\beta}{2}}}{(\log\lambda_n)^{\frac{1-\beta}{4}}} |A_{n,m}|.\nonumber 
\end{eqnarray}
Arguing as in the proof of Theorem \ref{thm30}, if $(\delta_n)$ contains a subsequence  $\left(\delta_{n_k}\right)$ such that $\big(\delta_{n_k}/(1+n_k)^{\frac{1}{\beta}-1}\big)$ is unbounded, then $\left(A_{n,m}\right)$ is unbounded and, consequently, $\left(B_{n,m}\right)$ is unbounded too. Thus \eqref{second}  holds. Suppose now that $\Delta_N =\frac{1}{2}+\varepsilon_N$, for a nonnegative sequence $\left(\varepsilon_N\right)$. Again as in the proof of Theorem \ref{thm30}, the sequence $|A_{n_k+N_k,n_k}|+|A_{n_k,n_k+N_k}|$ is unbounded and,  hence,  $\left(B_{n,m}\right)$ is unbounded too. This proves \eqref{third} and completes the proof.

{
\section{Final remarks}
{The following remark shows that the superior limit in condition \eqref{third} in Theorem \ref{thm30} can be replaced by a  supremum. This shows that in  the case $\beta=1$, Theorem \ref{thm30} and \cite[Theorem 1.1]{BDHK} are equivalent.
\begin{rem}\label{rem1}
Let $(\delta_n)$ be a sequence of real numbers such that $\left(\delta_n/(1+n)^{\frac{1}{\beta}-1}\right)\in\ell^\infty$. Let $N$ be a positive integer and $\delta=\frac{1}{2\beta(1+\beta)^{\frac{1}{\beta}}}$. The following conditions are equivalent 
\begin{enumerate}
\item $\displaystyle \underset{n}{\sup}\frac{1}{N(1+n)^{\frac{1}{\beta}-1}}\left|\underset{k=n+1}{\overset{n+N}{\sum}}\
\delta_k\right| <  \delta,$
\item $\displaystyle \underset{n}{\limsup}\frac{1}{N(1+n)^{{\frac{1}{\beta}-1}}}\left|\underset{k=n+1}{\overset{n+N}{\sum}}\ 
\delta_k\right| <  \delta.$
\end{enumerate}
\end{rem}
\begin{proof}
First, it is obvious that $$\underset{n}{\limsup}\frac{1}{N(1+n)^{\frac{1}{\beta}-1}}\left|\underset{k=n+1}{\overset{n+N}{\sum}}\
\delta_k\right|\leq \underset{n}{\sup}\frac{1}{N(1+n)^{\frac{1}{\beta}-1}}\left|\underset{k=n+1}{\overset{n+N}{\sum}}\
\delta_k\right|.$$ 

Conversely, let $\Gamma=\{\gamma_n\}$ be a $d-$separated sequence of complex numbers ordered in such a way that $|\gamma_n|\leq|\gamma_{n+1}|$. Set  $\gamma_n=\lambda_n\mathrm{e}^{\delta_n}\mathrm{e}^{i\theta_n}$. Suppose that 
$$\Delta_N:=\underset{n}{\limsup}\frac{1}{N(1+n)^{\frac{1}{\beta}-1}}\left|\underset{k=n+1}{\overset{n+N}{\sum}}\ 
\delta_k\right|<\delta.$$ Let $\varepsilon\in(0,\delta-\Delta_N)$. There exists $m\geq 0$ such that 
$$\sup_{n\geq m}\frac{1}{N(1+n)^{\frac{1}{\beta}-1}}\left|\sum_{k=n+1}^{n+N} 
\delta_k\right|\leq \Delta_N+\varepsilon <\delta.$$
Thus, $\cK_{\widetilde{\Gamma}}$ is a Riesz basis for $\cF^2_{\varphi}$, where $$\widetilde{\Gamma}:=\left\{\mathrm{e}^{\left(\frac{n+1}{1+\beta}\right)^{\frac{1}{\beta}}}\ :\ 0\leq n\leq m-1 \right\}\cup\left\{\gamma_n\in\Gamma\ :\ n\geq m\right\}.$$ Consequently, $\cK_\Gamma$ is also a Riesz basis for $\cF^2_{\varphi}$.
\end{proof}
}

\bibliographystyle{amsplain}

\end{document}